\documentclass[a4paper,11pt]{article}
\usepackage{graphicx}
\usepackage{amsfonts}
\usepackage{amsmath}
\usepackage[utf8]{inputenc}
\usepackage{url}
\usepackage{color}
\usepackage{amsthm}
\usepackage{amssymb}
\usepackage{enumerate}

 {
     \theoremstyle{plain}
     \newtheorem{assumption}{Assumption}
 }

\newtheorem {Proposition}{Proposition}[section]
\newtheorem {Lemma}[Proposition] {Lemma}
\newtheorem {Theorem}[Proposition]{Theorem}
\newtheorem {Corollary}[Proposition]{Corollary}
\newtheorem {Remark}[Proposition]{Remark}

\def\N{\mathbb{N}}

\def\R{\mathbb{R}}

\numberwithin{equation}{section}
\RequirePackage[OT1]{fontenc}

\begin{document}

\title{Central Limit Theorems for General Transportation Costs}


\author{Eustasio del Barrio$^{(1)}$\footnote{Research partially supported by FEDER, Spanish Ministerio de Econom\'ia y Competitividad, grant MTM2017-86061-C2-1-P and Junta de Castilla y Le\'on, grants VA005P17 and VA002G18.}, Alberto Gonz\'alez-Sanz$^{(2)}$, and Jean-Michel Loubes $^{(3)}$\footnote{Research partially supported by the AI Interdisciplinary Institute ANITI, which is funded by the French “Investing for the Future – PIA3” program under the Grant agreement ANR-19-PI3A-0004.}\\  $\,$ \\ 
{ $^{(1)(2)}$IMUVA, Universidad de Valladolid, Spain} \\ 
$^{(2)(3)}$IMT, Universit\'e de Toulouse
France\\ $\,$ \\ 
$^{(1)}$tasio@eio.uva.es \quad $^{(2)}$alberto.gonzalez sanz@math.univ-toulouse.fr \\ $^{(3)}$loubes@math.univ-toulouse.fr}
\maketitle

\begin{abstract}
We consider the problem of optimal transportation with general  cost between a empirical 
measure and a general target probability on $\mathbb{R}^d$, with $d\geq 1$. We extend results in \cite{BaLo} and prove asymptotic stability of both optimal transport maps and potentials for a large class of costs in $\R^d$. We derive a central limit theorem (CLT) towards a Gaussian distribution for the empirical transportation cost under minimal assumptions, with a  new proof  based on the Efron-Stein inequality and on the sequential compactness   of the closed unit ball in $L_2(\mathbb{P})$ for the weak topology.
\end{abstract}

\section{Introduction}

In the last few years new techniques based on the optimal transportation problem have become popular to handle statistical and machine learning problems over the space of probability distributions. Dealing with distributions has shed light on the need for probabilistic tools that are well adapted to the intrinsic geometry of the data, and the theory of optimal transport provides a natural framework to tackle such issues.  In particular the transportation cost distance is a convenient metric in many problems encountered in data science and the range of application fields is huge, including for instance computational statistics, biology, image analysis, economy, finance or fairness in machine learning. We refer for instance to \cite{SCHIEBINGER2019928}, \cite{HAL}, \cite{MAL-073}, \cite{courty}, \cite{bachoc2017gaussian},  \cite{gordaliza2019obtaining}, \cite{black2020} and references therein. Understanding the approximations done when dealing with empirical distributions  and providing  better controls on the asymptotic distribution of optimal transport cost is of importance for further research on this subject.

In all this work, we will be concerned with probabilities on  the measurable space $\R^d$, endowed with the Borel $\sigma$-field, denoted as $\mathcal{P}(\R^d)$. In this setting, the optimal transport problem is formulated as follows. Let $P,Q$ be probability measures in $ \mathcal{P}(\R^d)$ and $c:\R^d \times \R^d\rightarrow \R$ be a function referred to as the  \emph{cost}. We say that a  measurable map  $T:\R^d \rightarrow \R^d$ is an \emph{optimal transport map} from $P$ to $Q$ if it is a minimizer in the problem
\begin{align}\label{Monge1}
\mathcal{T}_c(P,Q):=\inf_{T:\, T_{\#}P=Q}\int_{\mathbb{R}^d} c(\textbf{x},T(\textbf{x})) d P(\textbf{x}), 
\end{align}
where the notation $T_{\#}P$ represents the \emph{push-forward} measure, that is, the measure such that for each measurable set $A$ we have $T_{\#}P(A):=P(T^{-1}(A))$. \\
This previous formulation of the problem is known as the Monge formulation and is closely related to the following problem known as the Kantorovich optimal transportation problem. \\
A probability measure $\pi \in \mathcal{P}(\R^d \times \R^d)$ is said to be an \emph{optimal transport plan for the cost $c$} between $P$ and $Q$ if it is a minimizer in the problem
\begin{align}\label{kant}
\mathcal{T}_c(P,Q)=\inf_{\gamma \in \Pi(P,Q)}\int_{\R^d\times \R^d} c(\textbf{x},\textbf{y}) d \pi(\textbf{x}, \textbf{y}),
\end{align}
where $\Pi(P,Q)$ is the set of probability measures $\pi \in \mathcal{P}(\R^d \times \R^d)$ such that $\pi(A\times \R^d)=P(A)$ and $\pi(\R^d \times B)=Q(B)$ for all $A,B$ measurable sets. We have used the same notation for the minimum value in both \eqref{Monge1} and \eqref{kant}, and this and the existence of optimal transport maps indeed hold for rather general costs, as shown in \cite{GaMc}, including the \emph{potential} costs $c_p(\textbf{x},\textbf{y})=|\textbf{x}- \textbf{y}|^p$, $p\geq 1$. We write in this case $\mathcal{T}_p(P,Q)$ for the minimal value in \eqref{Monge1} or \eqref{kant} and $\mathcal{W}_p(P,Q):=(\mathcal{T}_p(P,Q))^{1/p}$. Note that  $\mathcal{W}_p(P,Q)$ is  a distance on the subset in $\mathcal{P}(\R^d)$ of distributions with finite moment of order $p$, denoted as $\mathcal{P}^p(\R^d)$, referred to as the  $p-$Wasserstein or Monge-Kantorovich distance. This distance is closely related to the weak topology of $\mathcal{P}(\R^d)$, in the sense that  $P_n \stackrel{w}{\longrightarrow} P$ and $\int | \textbf{x}|^p dP_n(\textbf{x})\longrightarrow \int | \textbf{x}|^p dP(\textbf{x})$
is equivalent to $\mathcal{W}_p(P_n,P)\longrightarrow 0$.

For applications in statistics or machine learning, objects of main interest are $\mathcal{T}_c(P_n,Q)$ or $\mathcal{T}_c(P_n,Q_m)$, where 
$P_n$ (resp. $Q_m$) denotes the empirical measure on a sample $X_1,\ldots,X_n$ of i.i.d. observations with law $P$ (resp. a sample $Y_1,\ldots,Y_m$ i.i.d. $Q$). Early work on this topic, starting with \cite{AjtaiKomlosTusnady} (see also \cite{Talagrand1992,Talagrand,TalagrandYukich} and the more recent \cite{FouGui}), focused on the case $P=Q$ and provided rates of decay of 
$\mathcal{T}_p(P_n,P)$ (in fact, $\mathcal{T}_p(P_n,P)\to 0$ a.s. if $P$ has a finite moment of order $p$), which turns out to depend on the dimension of the sample space. The problem becomes simpler when this dimension is one, since, in this case, there is a common representation using the quantile function for all convex costs. This was exploited in \cite{BaGiMa} and
\cite{dBGU05} for proving distributional limit theorems for $\mathcal{T}_p(P_n,P)$, $p=1,2$.
We refer to \cite{BaLo} for a more detailed account about the history of the problem. The problem has received a renewed interest in the last 
few years, both in the setup $P=Q$ (see \cite{AmbrosioStraTrevisan, Ledoux2019, Talagrand2018}) or for general $P$ and $Q$ (see  \cite{SoMu} and \cite{TaSoMu} for finitely and countably supported probabilities, \cite{BaLo} for the case $p=2$ and general  probabilities and dimension and
\cite{BaGoLo, BeFoKl} for dimension $d=1$ and general costs).

In this paper we provide central limit theorems for $\mathcal{T}_c(P_n,Q)$ or $\mathcal{T}_c(P_n,Q_m)$ for general cost functions and general dimension, under minimal moment and regularity assumptions on $P$ and $Q$. Our contribution covers the strictly convex costs in \cite{GaMc} for which existence of the optimal transport is guaranteed. Strict convexity of the cost appears to be a minimal requirement 
for a general central limit theorem with a Gaussian limiting distribution. In fact, for the non strictly convex cost $p=1$ and in a univariate setup, \cite{BaGiMa} shows that $\{\sqrt{n}\mathcal{T}_{1}(P_n,P)\}_{n \in \N}$ converges to a non-Gaussian distribution under some regularity assumptions. Our moment assumptions improve upon those in \cite{BaLo}. The main result there is that, under mild regularity assumptions on $P$ and $Q$ (these are assumed to be absolutely continuous probabilities on $\mathbb{R}^d$ with convex supports)
\begin{equation}\label{quadratic1}
\sqrt{n}\left(\mathcal{T}_2(P_n,Q)-E \mathcal{T}_2(P_n,Q)\right)\stackrel{w}{\longrightarrow} N(0, \sigma^2_2(P,Q))
\end{equation}
provided that $P$ and $Q$ have finite moments of order $4+\delta$ for some $\delta>0$. We could take $Q=\delta_0$ (Dirac's measure on $0$) and see that in that case $\mathcal{T}_2(P_n,Q)=\frac 1 n \sum_{i=1}^n |X_i|^2$, and therefore, that a finite moment of order 4 is necessary (and sufficient in this case) for a CLT. One may wonder if \eqref{quadratic1} still holds under the weaker assumption of finite fourth moments. In fact, in dimension $d=1$, \cite{BaGoLo} proves that for $p>1$
\begin{equation}\label{generalpd1}
\sqrt{n}\left(\mathcal{T}_p(P_n,Q)- E\mathcal{T}_p(P_n,Q)\right)\stackrel{w}{\longrightarrow} N(0, \sigma^2_p(P,Q)), 
\end{equation}
assuming only finite moments of order $2p$ and continuity of the quantile function (an equivalent formulation of this last assumption is that the support of $Q$ is an interval, see Proposition A.7 in \cite{BobkovLedoux}). 

Similar results, but with quite a few more requirements on the regularity of the probabilities, are proved in \cite{BeFoKl}. 
The key to prove \eqref{quadratic1} and \eqref{generalpd1} (the same approach has been used in \cite{MeNi} to study the asymptotic behaviour of  entropically regularized Wasserstein distances) is a linearization technique based on the Efron-Stein inequality for variances coupled with stability results for optimal transportation \textit{potentials}.
For continuous costs the Kantorovich problem \eqref{kant} admits an equivalent dual form, namely,
\begin{align}\label{dual}
\mathcal{T}_c(P,Q)=\sup_{(f,g)\in \Phi_c(P,Q)}\int f(\textbf{x}) dP(\textbf{x})+\int g(\textbf{y}) dQ(\textbf{y}),
\end{align}
where $\Phi_c(P,Q)=\{ (f,g)\in L_1(P)\times L_1(Q): \ f(\textbf{x})+g(\textbf{x})\leq c(\textbf{x},\textbf{y}) \}$. It is said that  $\psi\in L_1(P)$ is an {optimal transport potential from $P$ to $Q$ for the cost $c$} if there exists $\varphi\in L_1(Q)$ such that the pair $(\psi, \varphi)$ solves \eqref{dual}.

The present contribution generalizes the sharp results in \cite{BaGoLo} to multivariate probabilities and to much more general costs than $c_p$. Also, our results cover those in \cite{BaLo}, improving them in the sense that here we do not require $P$ and $Q$ to have a convex support, but only a connected support with a negligible boundary. Furthermore, to avoid the technical need for stronger-than-necessary moment assumptions, in this work we describe a completely new tool to prove a central limit theorem for transportation costs. This approach can be summarized as follows. We try to show that the empirical transportation cost can be approximated by a linear term. The linearization error, say $R_n$ (see \eqref{eq_def_R} for details) has a variance that can be bounded using the Efron-Stein inequality. The upper bound is the expected value of a random variable, say $U_n$, which converges to 0 a.s. (this convergence follows from the stability results for optimal transport potentials), but one cannot conclude from this that $E(U_n)\to 0$ without further conditions (as in \cite{BaLo}, for instance). Yet, one can show that the sequence $EU_n$ is bounded and then Banach-Alaoglu theorem yields weak convergence in $L_2(\mathbb{P})$ of $U_n$ along subsequences. By taking Ces\`aro means we can go from weak to strong convergence and, with some additional work, to conclude that $\sqrt{n}(R_n-ER_n)\to 0$ in probability, which immediately yields a CLT. All the details are provided in Section \ref{sec:CTL} and in the proofs in the Appendix.

A second, relevant contribution in this paper are new results on the convergence and, in some sense, the uniqueness of both optimal transport potentials and maps between $P_n$ and $Q_n$ when these sequences converge weakly to some probabilities $P$ and $Q$. There is a large amount of literature working on these topics. Convergence of optimal maps is a topic of general interest, beyond our application to CLTs and results on this issue have a long history, tracing back at least to \cite{Cuestaetal97}. To our knowledge, interest on the convergence of potentials is more recent
and requires some additional guarantee on the uniqueness of the potentials.
Seminal  results on it can be found in Theorem 2.8 in \cite{BaLo} for the quadratic cost. Corollary 5.23 in \cite{Vi}  deals with this problem for general costs but one of both probabilities is supposed to be fixed. Some results are provided in Theorem 1.52. in \cite{San} when the involved probabilities are compactly supported.\\
\indent Then problem of uniqueness of optimal transport potentials is linked to the smoothness of the probabilities and also to the topology of their supports. For a probability $Q$ is the smallest closed set $R_Q$ such $Q(R_Q)=1$. Yet, we will use the notation 
\begin{align}\label{supp}
\text{supp}(Q):=\text{int}\left(R_{Q}\right)
\end{align}
for the interior of $R_Q$. Moreover, we say that a probability $Q$ has \textit{negligible boundary} if $\ell_d(R_{Q}\setminus \text{supp}(Q))=0$,
where $\ell_d$ denotes Lebesgue measure on $\R^d$. A probability with a convex support has a negligible boundary, but the condition is far from necessary. When the cost is of the form $c(\mathbf{x},\mathbf{y})=h(\mathbf{x}-\mathbf{y})$ with $h$ satisfying some regularity assumptions (see (A1)-(A3) and the related discussion in Section 2) and $Q$ has a density with respect to Lebesgue measure (in the sequel, when $Q\ll \ell_d$) and a connected support with negligible boundary then we prove (Corollary \ref{cor:uniqueness}) that optimal transport potentials are unique up to an additive constant (it is easy to see that this fails if $Q$ has a disconnected support). 
From this uniquenes we move on to give general stability results for optimal trasnport potentials under  only the following assumption
\begin{assumption}\label{assumptio:conv_pot}
$Q\in \mathcal{P}(\R^d)$ is such that $Q\ll \ell_d$ and has connected support with negligible boundary; $Q_n, P_n,P\in \mathcal{P}(\R^d)$ are such that $P_n\stackrel{w}\rightarrow P$, $Q_n\stackrel{w}\rightarrow Q$,
$$\mathcal{T}_c(P_n,Q_n)<\infty \ \ \text{and} \ \ \mathcal{T}_c(Q,P)<\infty ,$$ 
for a cost $c(\mathbf{x},\mathbf{y})=h(\mathbf{x}-\mathbf{y})$ with $h$ differentiable and satisfying (A1)-(A3).
\end{assumption}
If $\psi_n$ (resp. $\psi$) are the $c$-optimal transport potentials from $Q_n$ to $P_n$ (resp. from $Q$ to $P$), then we prove in Theorem~\ref{Theo:main2}  that
\begin{enumerate}
\item[\textit{(a)}] There exist constants $a_n\in \R$ such that $\tilde{\psi_n}:=\psi_n-a_n \rightarrow \psi$ in the sense of uniform convergence on the compacts sets.
\item[\textit{(b)}] 
For each compact $K\subset \text{supp}(Q)\cap \text{dom}(\nabla \psi)$
\begin{align*}
\sup_{\mathbf{x}\in K}\sup_{\mathbf{y}_n\in \partial^c \psi_n(\mathbf{x})}|\mathbf{y}_n-\nabla^c \psi(\mathbf{x}) |\longrightarrow 0,
\end{align*}
\end{enumerate}
where $\text{dom}(\nabla \psi)$ denotes the set of points where $\psi$ is differentiable.

The second main contribution of this paper is to provide a general result on the 
nature of the fluctuation of the empirical transportation cost around its expected value. More precisely, we consider the asymptotic behaviour of  $\{\sqrt{n}(\mathcal{T}_{c}(P_n,Q)-E\mathcal{T}_{c}(P_n,Q)) \}_{n \in \N}$. Our main result (Theorem~\ref{Theorem_vaiance_bound_general_sin_delta}), establishes the convergence 
\begin{align*}
\sqrt{n}\left(\mathcal{T}_c(P_n,Q)-E \mathcal{T}_c(P_n,Q)\right)\stackrel{w}{\longrightarrow} N(0, \sigma^2_c(P,Q)), 
\end{align*}
 with 
\begin{align*}
\sigma^2_c(P,Q):=\int{\varphi(\mathbf{x})^2}dP(\mathbf{x})-\left( \int{\varphi(\mathbf{x})}dP(\mathbf{x})\right)^2,
\end{align*}
where $\varphi$ is an optimal transport potential for the cost $c$, from $P$ to $Q$. This CLT holds assuming only that $c(\mathbf{x},\mathbf{y})=h(\mathbf{x}-\mathbf{y})$ with $h$ differentiable and satisfying (A1)-(A3) and $P, Q\in \mathcal{P}(\R^d)$ satisfying 
\begin{assumption}\label{assumption_dist}
$P\ll \ell_d$ and $Q\ll \ell_d$  have connected supports with negligible boundary; moreover 
\begin{align*}
\int h(2\mathbf{x})^2dP(\mathbf{x})<\infty,\quad \int h(-2\mathbf{y})^2dQ(\mathbf{y})<\infty,
\end{align*}
and
\begin{align*}
\begin{split}
&\inf_{q_1,q_2\in[1,\infty]:\ \frac{1}{q_1}+\frac{1}{q_2}=1}E| X_1-X'_1|^{2q_1}E\left(\int_{\R^d} |\nabla h(X_1-\textbf{y})|^{2q_2} d Q(\textbf{y})\right) <\infty,
\end{split}
\end{align*}
where $X_1$ has law $P$ and $X_1'$ is an independent copy of $X_1$.
\end{assumption}
We note that $ \sigma^2_c(P,Q)$ is well defined, in the sense that it does not depend on the chosen potential, which is proved to be unique up to an additive constant in Corollary~\ref{cor:uniqueness}.\\
The linearization technique that we use yields CLTs for the transportation cost under minimal assumptions. We discuss this with detail in the case of potential costs in Section 4. As a minor prize to pay, the approach does not yield moment convergence. We show in Theorem \ref{Theorem_vaiance_bound_general} that moment conevergence holds under some additional moment assumptions. Finally, we derive a CLT for the empirical transportation cost in a two-sample setup and a further CLT for the empirical $p$-Wasserstein distance. \vskip .1in
We end this Introduction with some details about our setup and notation. We assume all the involved random variables (we use this term for both $\R$ and $\R^d$-valued random elements) to be defined on a probability space $(\Omega, \mathcal{A},\mathbb{P})$. We write $L_2(\mathbb{P})$ for the Hilbert space of square integrable random variables on the former space. $\to_w$ denotes weak convergence of probability measures, while we write $\rightharpoonup$ for weak convergence (in the usual sense in Functional Analysis) in the space $L_2(\mathbb{P})$. At some points we write $A\subset \subset B$ to mean that there is some compact set, $K$, such that $A \subset K \subset B$.

\section{Preliminary results on optimal transport maps and potentials}\label{sec:c-con}
This section presents some results related to optimal transport potentials and maps for general costs. The main reference on the topic is \cite{GaMc}. {We give two main results, which are necessary tools for the study of stability in section 3: we prove uniqueness, up to an additive constant, of the optimal transport potential (Corollary~\ref{cor:uniqueness}) and a weak continuity result for a version of the optimal transport maps Lemma \ref{lemma:cont_c}.} 

We consider the optimal transport problem formulated in its dual form \eqref{dual}. Convexity plays a key role in the optimal transportation problem with quadratic cost. This idea can be adapted to general costs through the notion of $c$-concavity. Recall that  $f:\R^d\rightarrow \R\cup \{-\infty\}$ is said to be \textit{$c$-concave} 
if there exist a set  $\mathcal{T}\subset \R^d\times \R$ such that 
\color{black}
\begin{align}\label{eq:c_concave}
f(\mathbf{x})=\inf_{(\mathbf{y},\mathbf{t})\in \mathcal{T}}\{ c(\mathbf{x},\mathbf{y})-\mathbf{t} \}.
\end{align}
For a function $f:\R^d\rightarrow \R\cup \{-\infty\}$ the \emph{$c$-conjugate of $f$} (see \cite{GaMc}) is defined as
\begin{align}\label{eq:c_conj}
f^c(\mathbf{y})=\inf_{\mathbf{x}\in {\mathbb{R}^d}}\{ c(\mathbf{x},\mathbf{y})-f(\mathbf{x})\} \ \ \text{for all  }\mathbf{y}\in {\mathbb{R}^d}.
\end{align}
$c$-conjugation can be seen as a generalization of the Legendre's transform in convex analysis, see \cite{Rockafellar}.
Obviously, $f^c$ is $c$-concave and it is easy to check that its own $c$-conjugate, $f^{cc}$, satisfies $f^{cc}\geq f$, with equality if $f$ is $c$-concave. This means that we can restrict the collection of pairs $(f,g)$ in \eqref{dual} to pairs $(f, f^c)$, with $f$ $c$-concave, without changing the optimal value.

For a $c-$concave function $f:\mathbb{R}^d\rightarrow \R\cup \{-\infty\}$
the \textit{$c$-superdifferential} of $f$, $\partial^c f$, is the set of pairs $(\mathbf{x},\mathbf{y})\in \mathbb{R}^d\times \mathbb{R}^d$ such that 
$$f(\mathbf{z})\leq f(\mathbf{x})+[c(\mathbf{z},\mathbf{y})-c(\mathbf{x},\mathbf{y})]\quad \mbox{ for all } \mathbf{z}\in \mathbb{R}^d$$ 
(see, e.g., Definition 1.1 in \cite{GaMc}).
We write $\partial^c f(\mathbf{x})$ for the set of $\mathbf{y}$ such that $(\mathbf{x},\mathbf{y})\in\partial^c f$ and, more generally,
$\partial^c f(U)=\cup_{\mathbf{x}\in U}\partial^c f(\mathbf{x})$ for $U\subset \mathbb{R}^d$. Under mild assumptions (implied by (A1)-(A3) below; see Propositions C.3 and C.4 in \cite{GaMc}) $\partial^c f(\mathbf{x})$ is nonempty if $f$ is finite in a neighborhood of $\mathbf{x}$. When $\partial^c f(\mathbf{x})$ is a singleton we denote this point as $\nabla^c f(\mathbf{x})$. It is easy to see, for a $c$-concave function $f$, that
$f(\mathbf{x})+f^c(\mathbf{y})\leq c(\mathbf{x},\mathbf{y})$, with equality if and only if $\mathbf{y}\in \partial^c f(\mathbf{x})$. As a consequence of these key observations, $\pi\in\Pi(P,Q)$ is an optimal transport plan (a minimizer in \eqref{kant}) and the $c$-concave function $f$ is an optimal transport potential ($(f,f^c)$ is a maximizer in \eqref{dual}) if and only if $\pi$ is concentrated on the set $\partial^c f$.
This yields a characterization of optimal transport plans, {provided a maximizer in \eqref{dual} exists}. In that case we can get an equivalent
description of optimal transport plans in terms of cyclical monotonocity (see \cite{SmKn,Ruschendorf2}). 
A set $\Gamma\subset \mathbb{R}^d\times \mathbb{R}^d$ is said to be \textit{$c$-cyclically monotone} if for all $n\in \N$ and $\{(\mathbf{x}_k,\mathbf{y}_k)\}_{k=1}^{n} \subset \Gamma $ 
\begin{align}
	\sum_{k=1}^n c(\mathbf{x}_k,\mathbf{y}_k)\leq \sum_{k=1}^n c(\mathbf{x}_{\sigma(k)},\mathbf{y}_k),
\end{align}
for every permutation $\sigma$ in $\{1, \dots, n\}$. Optimal transport plans are supported in $c$-cyclically monotone sets (see Theorem \ref{Theo:GaMc} below). In the convex case (which corresponds to the quadratic cost $c_2(\mathbf{x},\mathbf{y})=|\mathbf{x}-\mathbf{y}|^2$) cyclically monotone sets are those contained in the subdifferential of a convex function and the subdifferential of a convex function is maximal cyclically monotone 
(this is known as Rockafellar's Theorem, see for instance,  \cite{Rockafellar2}). For general costs a similar result holds. We quote it for convenience in the next Lemma. A proof can be found in \cite{Ruschendorf} (Lemma 2.1). Note that Lemma~\ref{lem:Ruschendorf} is weaker than Rockafellar's Theorem for convex functions, since it does not claim that the set $\partial^c f$ is maximal.

\begin{Lemma}\label{lem:Ruschendorf}
If $c\geq 0$ is a continuous cost then a set $\Gamma\subset \mathbb{R}^d\times \mathbb{R}^d$ is $c$-cyclically monotone if and only if there exists a $c$-concave function $f$ such that $\Gamma \subset \partial^c f$.
\end{Lemma}

Existence of maximizing pairs in \eqref{dual} (which, as noted above, would yield a characterization of optimal transport plans) is not guaranteed
without some assumptions on the cost. Hence, we restrict our study to regular costs in the sense of~\cite{GaMc}, as follows: we will assume $c(\mathbf{x},\mathbf{y})=h(\mathbf{x}-\mathbf{y})$, where $h:\R^d\rightarrow [0, \infty)$ is a non negative function satisfying
\begin{enumerate}
	\item[(A1)] $h$ is strictly convex on $\R^d$,
	\item[(A2)] given a height $r\in \R^+$ and an angle $\theta \in (0,\pi) $, there exists some $M:=M(r, \theta)>0$ such that for all $|\mathbf{p} |>M$, one can find a cone 
	\begin{align*}
		K(r, \theta, \mathbf{z},\mathbf{p}):=\left\lbrace \mathbf{x}\in \R^d : | \mathbf{x}-\mathbf{p}|| \mathbf{z}|\cos(\theta/2)\leq \left< \mathbf{z},\mathbf{x}-\mathbf{p} \right>\leq  r| \mathbf{z}| \right\rbrace,
	\end{align*}
	with vertex at $\mathbf{p}$ on which $h$ attains its maximum at $\mathbf{p}$,
	\item[(A3)] $\lim_{|\mathbf{x} | \rightarrow \infty}\frac{h(\mathbf{x})}{|\mathbf{x} | }= \infty $.
\end{enumerate}

\begin{Remark}\label{Remark:p_satisfies}
The potential cost $c_p(\mathbf{x},\mathbf{y}):=| \mathbf{x}-\mathbf{y} |^p$ satisfies conditions {\em (A1)-(A3)} for $p> 1$, see {\em \cite{GaMc}}.
\end{Remark}

{
In the case of a quadratic cost the crucial step to turn the characterization of optimal transport plans into a characterization of 
optimal transport maps relies
on the fact that convex functions are locally Lipschitz, hence, by Rademacher's Theorem (see, e.g., Theorem 9.60 in \cite{RoWe}), they are differentiable at almost every point in the interior of their domain. For general costs convexity does not hold, but the Lipschitz property remains with great generality. In fact, if $g$ is a $c$-concave function then for every  $(\mathbf{a},\mathbf{b}),(\mathbf{x},\mathbf{y})\in \partial^c g$
we have
\begin{eqnarray}\label{eq:cota}
|g(\mathbf{x})- g(\mathbf{a})|&\leq &|c(\mathbf{x},\mathbf{y})-c(\mathbf{a},\mathbf{y})|+|c(\mathbf{x},\mathbf{b})-c(\mathbf{a},\mathbf{b})|.
\end{eqnarray}
When $c(\mathbf{x},\mathbf{y})=h(\mathbf{x}-\mathbf{y})$ with $h$ convex and differentiable, \eqref{eq:cota} implies that
\begin{eqnarray}\label{eq:cota2}
|c(\mathbf{x},\mathbf{y})-c(\mathbf{a},\mathbf{y})| &\leq&  |\mathbf{x}-\mathbf{a}|\left( |\nabla h (\mathbf{x}-\mathbf{y})|+  |\nabla h (\mathbf{a}-\mathbf{y})|\right).
\end{eqnarray}
As a consequence we obtain that
\begin{align*}
|g(\mathbf{x})- g(\mathbf{a})|\leq |\mathbf{x}-\mathbf{a}||\zeta(\mathbf{a},\mathbf{b},\mathbf{x},\mathbf{y})|, \ \ \text{for all} \ (\mathbf{a},\mathbf{b}),(\mathbf{x},\mathbf{y})\in \partial^c g,
\end{align*}
where $\zeta(\mathbf{a},\mathbf{b},\mathbf{x},\mathbf{y})$ is a continuous function (we recall that a differentiable convex function is, in fact, continuouly differentiable, see Corollary 25.5.1 in \cite{Rockafellar}). Ellaborating on these bounds it can be proved that
under (A1)-(A3) $c$-concave functions are locally Lipschitz, hence, differentiable at almost every point. For convenience we quote here a precise result (see Theorem 3.3 in \cite{GaMc}).

\begin{Lemma}\label{lem:Mccan}
Let $c(\mathbf{x},\mathbf{y})=h(\mathbf{x}-\mathbf{y})$ be a cost satisfying {\em (A1)-(A3)} and let $f$ be a $c$-concave function, then there exists a convex set $K\subset \R^d$ with interior $\Omega$ such that 
\begin{enumerate}
\item[(i)] $\Omega\subset \text{\em dom}(f)=\{\mathbf{x}: f(\mathbf{x})\in \R \}\subset K$,
\item[(ii)] $f$ is locally Lipschitz in $\Omega$.  
\end{enumerate}
\end{Lemma}

Now we can relate the shape of the gradient of a $c$-concave function to the shape of the $c$-superdifferential.
We write $h^*$ for the convex conjugate of $h$, namely, $h^*(\mathbf{y})=\sup_{\mathbf{x}} (\langle\mathbf{x},\mathbf{y}\rangle-h(\mathbf{x}) )$.
Then, if $f$ is $c$-concave (see Proposition 3.4 in \cite{GaMc}):
\begin{enumerate}
\item[\textit{a)}] the relation $\mathbf{s}(\mathbf{x})=\mathbf{x}-\nabla h^*(\nabla f(\mathbf{x}))$ defines a Borel function in the set where $f$ is differentiable, $\text{dom} (\nabla f)$,
\item[\textit{b)}] for all $\mathbf{x}\in \text{dom} (\nabla f)$ it holds that $\partial^c f(\mathbf{x})=\nabla^cf(\mathbf{x})=\{ \mathbf{s}(\mathbf{x})\}$,
\item[\textit{c)}] the set $\text{dom} ( f)\setminus\text{dom} (\nabla f)$ is of Lebesgue measure zero.
\end{enumerate}

Now, with all the ingredients above, a characterization of optimal transport plans and maps is given the next result, which summarizes
Theorems 1.2, 2.3 and 2.7 in \cite{GaMc}.

\begin{Theorem}\label{Theo:GaMc}
For any cost $c(\textbf{x},\textbf{y})=h(\textbf{x}-\textbf{y})$, satisfying {\em (A1)-(A3)}, and Borel probability measures $P,Q$ on $\R^d$:
\begin{enumerate}
\item[(i)] There exists at least an optimal transport plan. $\gamma\in \Pi(P,Q)$ is an optimal transport plan if and only if its support, $\text{\em Supp}(\gamma)$, is a $c$-cyclically monotone set, or, equivalently, if there exists a $c$-concave function $\psi$ such that 
$\text{\em Supp}(\gamma)\subset \partial^c\psi $. In this case $\psi$ is an optimal transport potential.
\item[(ii)]  If $P\ll \ell_d$, then there exists a unique optimal transport plan $\gamma:=(id\times T)\#P$, where $T(\textbf{x}):=\textbf{x}-\nabla h^*(\nabla \psi(\textbf{x}))=\nabla^c\psi (\textbf{x})$ is $P$-a.s. unique and the $c$-concave function $\psi$ is an optimal transport potential.
\end{enumerate} 
\end{Theorem}

The approach in this work to CLT's for the empirical transportation cost relies on the stability results for optimal transport potentials that we prove in Section \ref{sec:asymp}. There cannot be any result in that sense without some kind of uniqueness of this potential. Of course, a look at \eqref{dual} shows that if $\psi$ is an optimal transport potential and $C\in\mathbb{R}$ then $\psi+C$ is also an optimal transport potential. With  the next results we show that, under some minimal assumptions, the optimal transport potential is unique up to the addition of a constant. 

\begin{Lemma}\label{Lemma:iqual}
Let $\Omega\subset \R^d$ be an open, bounded convex set, $f: \Omega \longrightarrow \R$ be a Lipschitz function such that $\nabla f=\mathbf{0}$ almost everywhere in $\Omega$, then there exists a constant $C\in  \R$ such that $f=C$ in $L$.  
\end{Lemma}
\begin{proof}
This is a straightforward consequence of Poincar\'e's inequality in convex domains (see, e.g., Theorem 3.2. in \cite{AcDu}).
\end{proof}

\begin{Theorem}\label{theo:iqual}
Assume $c(\mathbf{x},\mathbf{y})=h(\mathbf{x}-\mathbf{y})$ satisfies {\em (A1)-(A3)} and $f_1,f_2$ are $c$-concave functions such that 
$\nabla f_1=\nabla f_2$ almost everywhere in the open connected set $\Omega$, then there exists a constant $C\in  \R$ such that $f_2=f_1+C$ in $\Omega$.  
\end{Theorem}
\begin{proof}
Assume $\mathbf{p}\in \Omega\subset\text{dom}(f_1)\cap \text{dom}(f_2)$. By Lemma \ref{lem:Mccan} $\phi, \psi$ are locally Lipschitz, hence, there exist $\epsilon_{\mathbf{p}}>0$ such that $f_1, f_2$ are Lipschitz in $B(\mathbf{p},\epsilon_{\mathbf{p}})$. Then the function $f_2-f_1$ satisfies the assumptions of Lemma~\ref{Lemma:iqual}. As a consequence, there exists $C_{\mathbf{p}}\in \R$ such that $f_2=f_1+C_{\mathbf{p}}$ in $\mathbb{B}(\mathbf{p},\epsilon_{\mathbf{p}})$ for each $\mathbf{p}\in L$. The proof will be complete if we show that the previous constant does not depend on $\mathbf{p}$. But this follows from connectedness of the $\Omega$, since if we set
$$ \Gamma:=\{ \mathbf{q}\in \Omega: C_{\mathbf{q}}=C_{\mathbf{p}}\}$$
then $\Gamma$ is obviously open and, by continuity, it is also closed, hence, $\Gamma=\Omega$.
\end{proof}

Let us assume now that $P$ and $Q$ are probabilities on $\mathbb{R}^d$ wit $P$ absolutely continuous and $\psi_1$, $\psi_2$ are optimal transport potentials. By Theorem \ref{Theo:GaMc} we have $\nabla h^*(\nabla \psi_1(\mathbf{x}))=\nabla h^*(\nabla \psi_2(\mathbf{x}))$ $P$-a.s.. If $h$ is differentiable then $\nabla h^*(\nabla \psi_1(\mathbf{x}))=\nabla h^*(\nabla \psi_2(\mathbf{x}))=\mathbf{y}$ implies 
$\nabla \psi_1(\mathbf{x})=\nabla \psi_2(\mathbf{x})=\nabla h(\mathbf{y})$. Hence, $P$-a.s., $\nabla \psi_1(\mathbf{x})=\nabla \psi_2(\mathbf{x})$. If, additionally, $P$ is supported in an open, connected set we can apply Theorem \ref{theo:iqual} and conclude that $\psi_2=\psi_1+C$ on the support of $P$ for some constant $C\in\mathbb{R}$. This proves the following uniqueness result for optimal transport potentials.

\begin{Corollary}\label{cor:uniqueness}
If $c(\mathbf{x},\mathbf{y})=h(\mathbf{x}-\mathbf{y})$, where $h$ is differentiable and satisfies {\em (A1)-(A3)}, $P\ll \ell_d$ and is supported on an open, connected set, $A$, and $\psi_1,\psi_2$ are optimal transport potentials from $P$ to $Q$ for the cost $c$, then, there exists a constant, $C\in\mathbb{R}$, such that $\psi_2(\mathbf{x})=\psi_1(\mathbf{x})+C$ for every $\mathbf{x}\in A$.
\end{Corollary}

In the next section we will state and prove results related to the stability of optimal transport maps and potentials, namely, we will prove convergence in different senses of optimal transport potentials ($\varphi_n$) or maps ($\nabla^c \varphi_n$) from $P_n$ to $Q_n$ under the assumption that 
(at least) $P_n\to_w $ and $Q_n\to Q$. Results of this kind have a long history, tracing back at least to \cite{Cuestaetal97} for the case of optimal transport maps under quadratic costs. Stability of the potentials is crucial for the Efron-Stein approach to CLTs in~\cite{BaLo} or in section 4 in this paper, and has only been investigated recently. For smooth probabilities, optimal transport potentials are a.s. differentiable, and there is a simple relation between their gradients and the optimal transport maps, as noted above. Hence, it is natural to try to go from stability results for optimal maps to stability results for optimal potentials. We should note, additionally, that the points of nondifferentiability of the potentials are those points in which the superdifferentials are not singletons and that, for this reason, the better way to deal with stability of the optimal plans is to think of them as multivalued maps ($\mathbf{x}\mapsto \partial^c \varphi_n(\mathbf{x})\subset \mathbb{R}^d)$ or, equivalently, as subsets ($\partial^c \varphi_n=\{(\mathbf{x},\mathbf{y})\in\mathbb{R}^d\times \mathbb{R}^d:\, \mathbf{y}\in\partial^c \varphi_n(\mathbf{x})\}$, the \textit{graph} of $\partial^c \varphi_n$).
The notion of convergence that fits our goals is the commonly called Painlev\'e-Kuratowski convergence (see \cite{RoWe}), which is defined as follows: for a sequence $\{ \Gamma_n\}_{n\in \N}$ of subsets of $\R^{m}$
\begin{itemize}
\item the outer limit, $\lim\sup_n \Gamma_n$, is the set of $\mathbf{x}\in \R^{m}$ for which there exists a sequence $\{\mathbf{x}_{n}\}$ with $\mathbf{x}_n\in \Gamma_n$ such that there exists a subsequence which converges to $\mathbf{x}$,
\item the inner limit, $\lim \inf _n \Gamma_n$,  is the set of $\mathbf{x}\in \R^{m}$ for which there exists a sequence $\{\mathbf{x}_{n}\}$ with $\mathbf{x}_n\in \Gamma_n$ which converges to $\mathbf{x}$.
\end{itemize}
When the outer and inner limit sets are equal the sequence is said to converge in the Painlev\'e-Kuratowski sense and the common set is the limit. This notion of convergence is automatically transferred easily to multivalued maps. In this case $\{T_n\}_{n\in \N}$, where $T_n:\R^d \rightarrow 2^{\R^d}$, is said to \emph{converge graphically} to another multivalued map $T$ if 
$$\text{Gph}(T_n):=\{(\mathbf{x}, \mathbf{y}): \mathbf{y}\in T_n(\mathbf{x}) \}\to \text{Gph}(T)$$
in the Painlev\'e-Kuratowski sense. A very convenient feature of the Painlev\'e-Kuratowski sense is that sequential compactness can be easily described in terms of a simple condition. To be precise, it is said that a sequence of sets $\Gamma_n\subset \mathbb{R}^d$, $n\geq 1$, does not \textit{escape to the horizon} if there exist $\epsilon>0$ and some subsequence $\{n_j\}$ such that $\Gamma_{n_j}\cap B(\mathbf{0}, \epsilon)\neq\emptyset$ for all $j\geq 1$. For convenience we quote next a version of Theorem 4.18 in \cite{RoWe}.
\begin{Theorem}\label{Teo:Horizon}
Let $\{ \Gamma_n\}_{n\geq 1}$ be a sequence of subsets of $\R^{m}$ that does not escape to horizon. Then there exists a subsequence $\{n_{j_k}\}$ and a nonempty subset $\Gamma \subset {\R^{m}}$ such that  
$$\Gamma_{n_{j_k}}\longrightarrow \Gamma,  \ \ \text{in the sense of Painlev\'e-Kuratowski.}$$
\end{Theorem}

In the next theorem we show that when a sequence of $c$-ciclically monotone sets converges in the sense of Painlev\'e-Kuratowski to a set, then it is also $c$-ciclically monotone, generalizing the result for classical convexity in \cite{RoWe}.

\begin{Lemma}\label{lemma:convergence_c}
Assume $c$ is a continuous cost function and $\{ \Gamma_n\}_{n\in \N}\subset \R^{2d}$ is a sequence of $c$-ciclically monotone sets. If $\Gamma_n\to \Gamma$ in the sense of Painlev\'e-Kuratowski, then $\Gamma$ is also $c$-ciclically monotone.
\end{Lemma}
\begin{proof}
We consider $\{(\mathbf{x}_k,\mathbf{y}_k)\}_{k=1}^{m} \subset \Gamma $. For each pair $(\mathbf{x}_k,\mathbf{y}_k)$ there exists a sequence $(\mathbf{x}_k^n,\mathbf{y}_k^n)\in \Gamma_n$ such that $(\mathbf{x}_k^n,\mathbf{y}_k^n)\rightarrow (\mathbf{x}_k,\mathbf{y}_k) $ as $n\to \infty$. 
Since $\Gamma_n$ is $c$-ciclically monotone,
\begin{align*}
	\sum_{k=1}^m c(\mathbf{x}^n_k,\mathbf{y}^n_k)\leq \sum_{k=1}^m c(\mathbf{x}^n_{\sigma(k)},\mathbf{y}^n_k),
\end{align*}
for every permutation $\sigma$ of $\{1, \dots, m\}$. Continuity of $c$ guarantees that
\begin{align*}
\sum_{k=1}^m c(\mathbf{x}_k,\mathbf{y}_k)\leq \sum_{k=1}^m c(\mathbf{x}_{\sigma(k)},\mathbf{y}_k).
\end{align*}
\end{proof}
Combining the last two results we see that if a sequence of $c$-superdifferen\-tials does not escape to the horizon, then there exists a converging subsequence to a set and this set is also $c$-ciclically monotone. \\ \\
We finish the section with a weak continuity result for the multivalued map $\partial^c \psi$, which will be very useful in the following section.{
\begin{Lemma}\label{lemma:cont_c}
Assume $c(\mathbf{x},\mathbf{y})=h(\mathbf{x}-\mathbf{y})$ with $h$ {satisfying {\em (A1)-(A3)}}. Let $f$ be a $c-$concave function and $\mathbf{x}\in \text{\em dom}(\nabla^cf)$. Then for each sequence $\mathbf{x}_n\rightarrow \mathbf{x}$ and $\mathbf{y}_n\in \partial^c f(\mathbf{x}_n)$ we have that $\mathbf{y}_n\rightarrow \nabla^c f(\mathbf{x})$.
As a consequence, for each $\epsilon>0$ there exists some $\delta>0$ such that $\partial^c f(B(\mathbf{x},\delta))\subset B(\nabla^cf(\mathbf{x}), \epsilon)$.
\end{Lemma}
\begin{proof}
Let $(\mathbf{x}_n, \mathbf{y}_n)$ be as in the statement. Then for every $\mathbf{z}\in \R^d$ we have
\begin{align}\label{eq:take_limits11}
f(\mathbf{z})\leq f(\mathbf{x}_n)+[c(\mathbf{z},\mathbf{y}_n)-c(\mathbf{x}_n,\mathbf{y}_n)].
\end{align}
Since $\psi$ is differentiable at $\mathbf{x}$, it is bounded in a neighborhood of $\mathbf{x}$, say $U$, which can choose to be compact. By Proposition C.4 in \cite{GaMc} $\partial^c f(U)$ is bounded. Hence, the sequence $\mathbf{y}_n$
must be bounded and, taking subsequences if necessary, we can assume that it is convergent.
Taking limits in \eqref{eq:take_limits11} and noticing that $f$ is continuous in its domain we get the first conclusion. To check the second claim, assume it is false. Then we can choose some $\epsilon>0$ such that for each $n\in \N$ there exists $|\mathbf{x}_n-\mathbf{x}|\leq \frac{1}{n} $ and some $\mathbf{y}_n\in \partial^cf(\mathbf{x}_n)$ with $|\mathbf{y}_n-\nabla^cf(\mathbf{x})|>\epsilon$. To conclude note that the sequences $\{\mathbf{x}_n\}_{n \in \N}$ and $\{\mathbf{y}_n\}_{n \in \N}$ lead to a contradiction with the first assertion.
\end{proof}
}

\section{Stability of Optimal Transport Potential and Map Under General Costs}\label{sec:asymp}

The main goal of this section is to prove a general result (Theorem~\ref{Theo:main2}) on the stability of optimal maps and potentials for a very large class of costs, using the tools presented in section \ref{sec:c-con}. 
The path to this main result starts by proving stability along subsequences of the $c$-superdifferentials of optimal transport potentials
(Lemma \ref{Teo:main1}), extending a similar result  in \cite{BaLo} for the particular setup of classical convexity. We prove then (Lemmas \ref{Lem:Boundness} and \ref{Lem:technic3}) a uniform boundedness result which, once the potentials are conveniently fixed at a convenient point (see \eqref{potential.fixing} below) allows to prove the anticipated stability result. For the sake of readability we present here the results and defer most of the proofs to the Appendix.

The first step in the plan above is this result on the stability of $c$-superdiffer\-en\-tials.
\begin{Lemma}\label{Teo:main1}
Let $Q\in \mathcal{P}(\mathbb{R}^d)$ be such that $Q\ll \ell_d$ and has connected support and negligible boundary. Let $Q_n, P_n,P\in \mathcal{P}(\mathbb{R}^d)$ be such that $P_n\stackrel{w}\rightarrow P$, $Q_n\stackrel{w}\rightarrow Q$ and
$$\mathcal{T}_c(P_n,Q_n)<\infty \ \ \text{and} \ \ \mathcal{T}_c(Q,P)<\infty , \ \ \text{for all $n \in \N$},$$ 
for a cost $c(\mathbf{x},\mathbf{y})=h(\mathbf{x}-\mathbf{y})$ with $h$ differentiable and satisfying {\em (A1)-(A3)}. If $\psi_n$ (resp. $\psi$) are the optimal transport $c$-potentials from $Q_n$ to $P_n$ (resp. from $Q$ to $P$), then there exists a cyclically monotone set $\Gamma$ such that 
\begin{align}\label{convergence_subs}
\partial^c \psi_n \rightarrow \Gamma\subset \partial^c\psi 
\end{align}
in the sense of Painlev\'e-Kuratowski along subsequences. Moreover, if $\mathbf{x}\in \text{\em dom}(\nabla \psi)$, then $(\mathbf{x},\nabla^c\psi(\mathbf{x}))\in\Gamma$.
\end{Lemma}

{In our next results we pay attention to the optimal transportation potentials, $\psi_n$, which are well defined, under the assumptions of
Corollary \ref{cor:uniqueness}, up to the addition of a constant. The possibility of arbitrarily choosing that constant could lead to some difficulties that we can avoid fixing it as follows. We choose some $\mathbf{p}_0\in \text{dom}(\nabla \psi)\cap \text{supp} (Q)$ and assume 
\begin{equation}\label{potential.fixing}
\psi(\mathbf{p}_0)=0\quad \text{ and } \quad \psi_n(\mathbf{p}_0)=0 \text{ for large } n.
\end{equation}
Of course, we can ensure that the potential $\psi$ vanishes at any $\mathbf{p}_0$ where it is finite by taking $\tilde{\psi}(\mathbf{x})={\psi}(\mathbf{x})- {\psi}(\mathbf{p}_0)$. Under the assumptions of Lemma \ref{Teo:main1} (see the proof for further details) we must
have $\mathbf{p}_0\in \text{dom}(\nabla \psi_n)$ for large enough $n$, hence, $\mathbf{p}_0\in \text{dom}(\psi_n)$ and we can choose the potentials as in \eqref{potential.fixing}.
}

{Next, we present two technical lemmas in which the assumptions (A2) and (A3) play the main roles. These results, crucial in the proof of Theorem~\ref{Theo:main2}, are proved ellaborating on the arguments in \cite{GaMc} to prove that a $c$-concave function is locally Lipschitz. 
The geometric interpretation of these results is shown in Figure~\ref{Fig_geometric}. Lemma~\ref{Lem:Boundness} shows that for any point $\mathbf{p}$ for which the boundedness condition fails, there is a hyperplane $H$ passing trough $\mathbf{p}$ and splitting the space into two parts such that in one of both, the grey one in Figure~\ref{Fig_geometric}, this property holds for any other point. 
}
\begin{Lemma}\label{Lem:Boundness}
Under the same assumptions as in Lemma~\ref{Teo:main1}, let $\mathbf{p}\in \R^d$ be such that there exists a sequence $\{ \mathbf{p}_n\}_{n\in \N}\subset \R^d$ such that $\mathbf{p}_n\rightarrow \mathbf{p}$ and $\psi_n(\mathbf{p}_n)$ is not bounded. Then there exists $\mathbf{z}\in \R^d$ such that, for every bounded sequence $\{ \mathbf{x}_n\}_{n\in \N}\subset\subset \{\mathbf{x}:\left\langle \mathbf{z},\mathbf{x}-\mathbf{p}\right\rangle>0 \}$, the sequence $\psi_n(\mathbf{x}_n)$ is not bounded.
\end{Lemma}

\begin{figure}
\includegraphics[width=8cm]{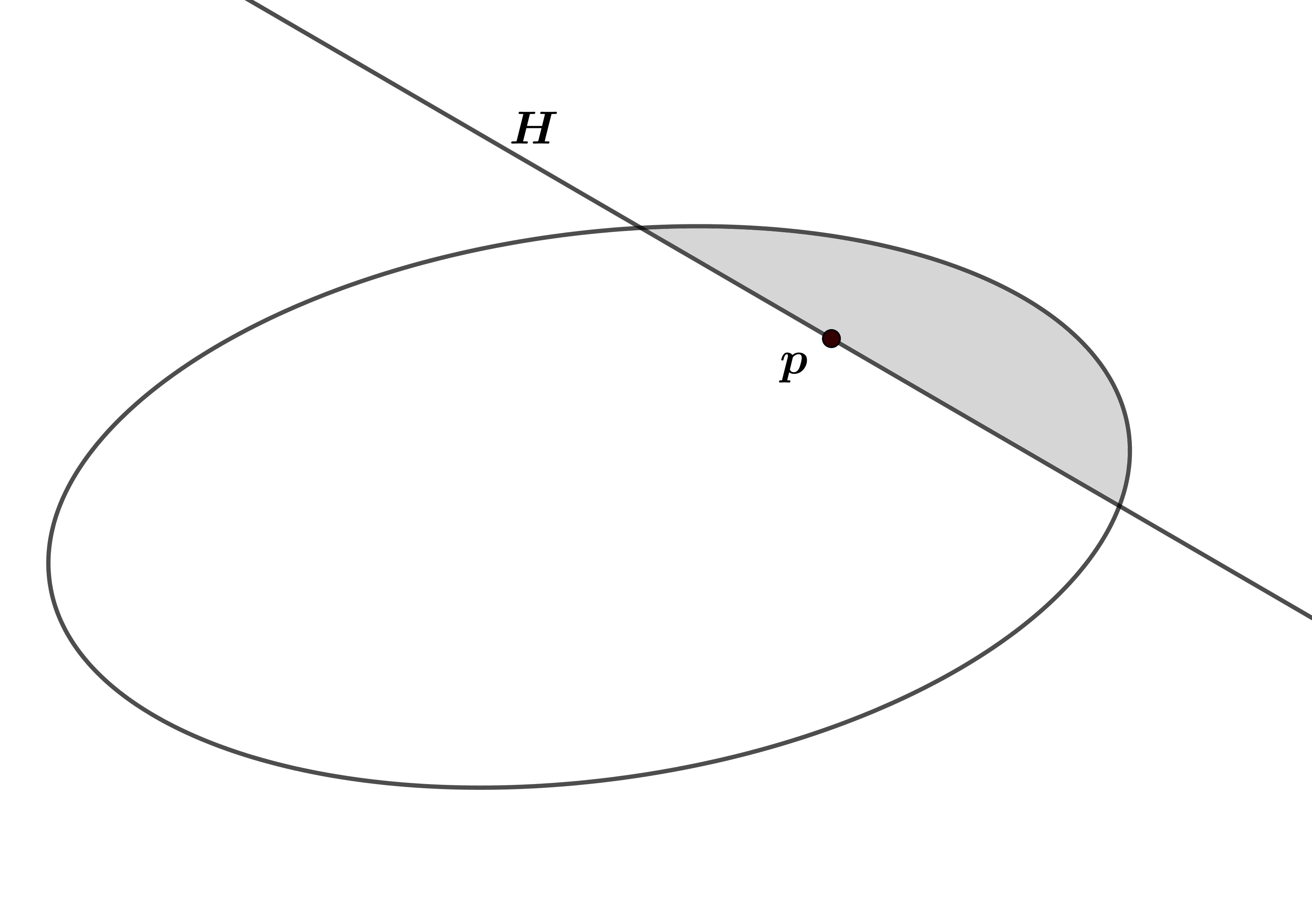}
\centering 
\caption{Geometric interpretation of Lemma~\ref{Lem:Boundness} and Lemma~\ref{Lem:technic3}.}
\label{Fig_geometric}
\end{figure}
Lemma \ref{Lem:Boundness} is the key to the next technical result, which proves boundedness of both $\bigcup_{k\in \N}\psi_{n_k}(K)$ and $\bigcup_{k\in \N}\partial^c\psi_{n_k}(K)$ for compact $K \subset \text{Supp}(Q)$.
\begin{Lemma}\label{Lem:technic3}
Let $P,Q,P_n, Q_n$ be probability measures satisfying the assumptions of Lemma~\ref{Teo:main1}. Assume that $\mathbf{p}_0\in\text{\em Supp} (Q)$ and $\psi_n(\mathbf{p}_0)\to 0$. Then for each compact $K \subset \text{\em Supp}(Q)$ there exists a subsequence $\{ \psi_{n_k}\}_{k\in \N}$ such that $\bigcup_{k\in \N}\psi_{n_k}(K)$ and $\bigcup_{k\in \N}\partial^c\psi_{n_k}(K)$ are bounded sets. 
\end{Lemma}
Now, as an application of the uniform boundedness results in Lemma \ref{Lem:technic3}, we are ready to apply the classical  Arzel\`a-Ascoli theorem to prove of the main theorem of the section.
\begin{Theorem}\label{Theo:main2}
Let $Q\in \mathcal{P}(\R^d)$ be such that $Q\ll \ell_d$ and has a connected support with negligible boundary. Assume $Q_n, P_n,P\in \mathcal{P}(\R^d)$ are such that $P_n\stackrel{w}\rightarrow P$, $Q_n\stackrel{w}\rightarrow Q$ and
$$\mathcal{T}_c(P_n,Q_n)<\infty \ \ \text{and} \ \ \mathcal{T}_c(Q,P)<\infty $$ 
for a cost $c(\mathbf{x},\mathbf{y})=h(\mathbf{x}-\mathbf{y})$, with $h$ differentiable and satisfying {\em (A1)-(A3)}.  If $\psi_n$ (resp. $\psi$) are optimal transport potentials from $Q_n$ to $P_n$ (resp. from $Q$ to $P$) for the cost $c$. Then:
\begin{enumerate}
\item[(i)] There exist constants $a_n\in \R$ such that $\tilde{\psi_n}:=\psi_n-a_n \rightarrow \psi$ in the sense of uniform convergence on the compacts sets of $\text{\em Supp}(Q)$.
\item[(ii)] For each compact $K\subset\subset \text{\em Supp}(Q)\cap \text{\em dom}(\nabla \psi)$
\begin{align}\label{eq:convergence_nabla}
\sup_{\mathbf{x}\in K}\sup_{\mathbf{y}_n\in \partial^c \psi_n(\mathbf{x})}|\mathbf{y}_n-\nabla^c \psi(\mathbf{x}) |\longrightarrow 0.
\end{align}
\end{enumerate}

\end{Theorem}

We note that Theorem~\ref{Theo:main2} generalizes Theorem 2.8 in \cite{BaLo} to a more general class of costs. Moreover, it also generalizes the results of stability of optimal transport maps, as Corollary 5.23 in \cite{Vi}. An important improvement of Theorem 1.52. in \cite{San} is obtained since we do not require a compact assumption. Finally we will see in the following sections that it is an useful tool to prove a Central Limit Theorem for general Wasserstein distances.

Under stronger assumptions on the way that $P_n$ approaches $P$ and $Q_n$ approaches $Q$ it is possible to prove $L_2$ convergence of the potentials. We show this next for potential costs. We recall that the hypothesis of Corollary~\ref{Coro:main2} are fulfilled when we have weak convergence $P_n\stackrel{w}{\rightarrow} P$, $Q_n\stackrel{w}{\rightarrow}Q$ plus convergence of moments of order $2p$,
$$\int |\mathbf{x} |^{2p} dP_n(\mathbf{x})\longrightarrow \int |\mathbf{x} |^{2p} dP(\mathbf{x}), \ \ \int |\mathbf{y} |^{2p} dQ_n(\mathbf{y})\longrightarrow \int |\mathbf{y} |^{2p} dQ(\mathbf{y}).$$

\begin{Corollary}\label{Coro:main2}
Let $Q\in \mathcal{P}_{2p}(\R^d)$ be such that $Q\ll \ell_d$ and has connected support with negligible boundary. Assume $P_n,P\in \mathcal{P}(\R^d)$ are such that 
\begin{align*}
\mathcal{T}_{2p}(P_n,P)\rightarrow 0.
\end{align*}
If $\psi_n$ (resp. $\psi$) are optimal transport potentials from $Q$ to $P_n$ (resp. from $Q$ to $P$) for the cost $c_p(\mathbf{x},\mathbf{y})=|\mathbf{x}-\mathbf{y}|^p $ and $p>1$, then there exist constants $a_n\in \R$ such that $\tilde{\psi_n}:=\psi_n-a_n\rightarrow \psi$ in the sense of $L^2(Q)$.
\end{Corollary}
\begin{proof}
We can apply Theorem~\ref{Theo:main2} to see that there exist constants $a_n\in \R$ such that $\tilde{\psi_n}=\psi_n-a_n\rightarrow \psi$ 
and $\nabla^c \psi_n \to \nabla^c \psi$ $Q$-a.s.. We note also that the assumption implies that $\int |\nabla^c \psi_n |^{2p}dQ\to 
\int |\nabla^c \psi |^{2p}dQ$ and, therefore, $\int |\nabla^c \psi_n -\nabla^c \psi|^{2p}dQ\to 0$. In particular $|\nabla^c \psi_n|^{2p}$ is $Q$-uniformly integrable.
We relabel the potentials and write $\psi_n$ instead of $\tilde{\psi}_n$ and assume (with no loss of generality) that $\psi(\mathbf{x}_0)=\psi_n(\mathbf{x}_0)=0$ for some $\mathbf{x}_0\in \text{Supp}(Q)\cap \text{dom}(\nabla\psi)$. 
To conclude, it suffices to show that $\psi_n^2$ is $Q$-uniformly integrable. To check this we set $\mathbf{y}_0=\nabla^c \psi(\mathbf{x}_0)$, take $\mathbf{y}_n\in\partial^c \psi_n(\mathbf{x}_0)$ and recall that, by Theorem~\ref{Theo:main2} $\mathbf{y}_n\to \mathbf{y}_0$. 
Now, we observe that
\begin{equation}\label{unif.int1}
\psi_n(\mathbf{x})\leq \psi_n(\mathbf{x}_0)+|\mathbf{x}-\mathbf{y}_n|^p-|\mathbf{x}_0-\mathbf{y}_n|^p\leq |\mathbf{x}-\mathbf{y}_n|^p
\end{equation}
for every $\mathbf{x}$. Similarly,
$$\psi^c_n(\mathbf{y})\leq \psi^c_n(\mathbf{y}_n)+|\mathbf{y}-\mathbf{x}_0|^p-|\mathbf{y}_n-\mathbf{x}_0|^p=|\mathbf{y}-\mathbf{x}_0|^p$$
for every $\mathbf{y}$. 
Since $Q$-a.s. we have $\psi_n(\mathbf{x})+\psi^c_n(\nabla^c \psi_n( \mathbf{x}))=|x-\nabla^c \psi_n( \mathbf{x})|^p$, we conclude that 
$$\psi_n(\mathbf{x})\geq |\mathbf{x}-\nabla^c \psi_n( \mathbf{x})|^p - |\nabla^c \psi_n( \mathbf{x})-\mathbf{x}_0|^p.$$
This last bound together with \eqref{unif.int1} shows that $\psi_n^2$ is $Q$-uniformly integrable and completes the proof.
\end{proof}

\section{Central Limit Theorem and Variance Bounds}\label{sec:CTL}
\subsection{One-sample case}\label{One_sam}
 Let $P\in \mathcal{P}(\R^d)$ and for each $n\in \N$ let $X_1,\dots, X_n$ denote a a sample of independent random variables with distribution  $P$. Consider also the correspondent empirical measure $P_n:=\frac{1}{n}\sum_{k=1}^{n}\delta_{X_k}.$ We are interested in the behavior of the sequence $\{	\sqrt{n}\left(\mathcal{T}_p(P_n,Q)-E \mathcal{T}_p(P_n,Q)\right)\}_{n\in \N}$. \\
We will prove first tightness of this sequence from a suitable variance bound, following  similar arguments as those in \cite{BaLo}. We recall the Efron-Stein inequality and refer for further details to  Chapter 3.1 in~\cite{Boucheron}. Let $(X'_1,\dots, X'_n) $ be an independent copy of $(X_1,\dots, X_n )$, set $Z:=f(X_1,\dots, X_n )$ and for each $i\in \{1, \dots, n \}$ denote $$Z'_i:=f(X_1,\dots,X_{i-1}, X'_{i}, X_{i+1},\dots, X_n ).$$ The Efron-Stein inequality states then that
$$ \text{Var}(Z)\leq  \frac{1}{2}\sum_{i=1}^n E(Z-Z_i)^2=\sum_{i=1}^n E(Z-Z'_i)_+^2.$$
Note that when $X_1,\dots, X_n $ are i.i.d, the inequality can be written as
$$ \text{Var}(Z)\leq  \frac{n}{2}E(Z-Z'_i)^2=n E(Z-Z'_i)_+^2.$$
In this work we present a general bound for the variance of $\mathcal{T}_c(P_n,Q)$ assuming only that one of both probabilities is absolutely continuous with respect to Lebesgue measure and assuming also that the cost is convex. We note that for $X$ with law $P$ the set of points where $h(X-\cdot)$ is not differentiable is a set of Lebesgue measure $0$, hence if $Q\ll \ell_d$ then it is differentiable $Q$-a.s.. As a consequence $\nabla h(X-\textbf{y})$ is well defined $Q$- a.s., and also $E|\nabla h(X-Y)|^{2q_2}$ in the next statement. 
\begin{Lemma}\label{Lemma:var_bound}
Assume $c(\mathbf{x},\mathbf{y})=h(\mathbf{x}-\mathbf{y})$, with $h$ satisfying (A1)-(A3).  Let $P, Q\in \mathcal{P}(\R^d)$ be such that $Q\ll \ell_d$. Assume $X,X',Y$ are independent random variables with $X\sim P$, $X'\sim P$ and $Y\sim Q$. Then 
\begin{equation}\label{eq:variance_bound}
n\text{\em Var}(\mathcal{T}_c(P_n,Q))\leq \inf_{(q_1,q_2)\in \alpha}\Big[\left(E| X-X'|^{2q_1}\right)^{\frac{1}{q_1}} \left(E|\nabla h(X-Y)|^{2q_2}\right)^{\frac{1}{q_2}}\Big],
\end{equation}
where $\alpha=\{(q_1,q_2):\, q_i\in[1,\infty], \frac{1}{q_1}+\frac{1}{q_2}=1\}$.
\end{Lemma}
We remark that assumptions (A1)-(A3) are only used in Lemma~\ref{Lemma:var_bound} to ensure the existence of an optimal transport map.
\begin{Remark}\label{remark_var}
As a consequence of Lemma~\ref{Lemma:var_bound}, under the same assumptions, if
\begin{align}\label{eq:tight_cond_1}
\inf_{(q_1,q_2)\in \alpha}\Big[\left(E| X-X'|^{2q_1}\right)^{\frac{1}{q_1}} \left(E|\nabla h(X-Y)|^{2q_2}\right)^{\frac{1}{q_2}}\Big]<\infty,
\end{align}
then the sequence $\{	\sqrt{n}\left(\mathcal{T}_c(P_n,Q)-E \mathcal{T}_c(P_n,Q)\right)\}_{n\in \N}$ is tight. 
\end{Remark}
We show next that we can replace assumption~\eqref{eq:tight_cond_1} with a simpler version in the case of potential costs. It should be noted that absolute continuity of $Q$ is not needed for the following result.
\begin{Corollary}\label{Lemma_tigh_p}
If $c(\textbf{x},\textbf{y})=|\textbf{x}-\textbf{y}|^p$ and $p>1$ then 
$$n\text{\em Var}(\mathcal{T}_p(P_n,Q))\leq  \left(E| X-X'|^{2p}\right)^{\frac{1}{p}} \left(p E|X-Y|^{2p}\right)^{\frac{p}{p-1}}$$
\end{Corollary}

\begin{proof}
We assume that the right hand side in the last bound is finite (there is nothing to prove otherwise). 
Since $|\nabla h(X_1-\textbf{y})|=p|X_1-\textbf{y}|^{p-1}$, the result follows by taking $q_1=p$, $q_2=\frac p{p-1}$ in \eqref{eq:variance_bound} 
if $Q\ll \ell_d$.For general $Q$ we can take random variables $Y\sim Q$, $Y_m\sim Q_m$, $m\in \N$ with $Q_m\ll \ell_d$ and  $E|Y_m-Y|^{2p}\to 0$. 
Without loss of generality we can assume that $(X,X')$ is independent of $(Y,\{Y_m\}_{m\geq 1}$). For fixed $n \in \N$ we have that $\mathcal{T}_p(P_n,Q_m)$ converges to  $\mathcal{T}_p(P_n,Q)$ a.s. as $m\to\infty$. Also, for each $m\in \N$, we have
\begin{align*}
n\text{Var}(\mathcal{T}_p(P_n,Q_m))\leq  \left(E| X-X'|^{2p}\right)^{\frac{1}{p}} \left(p E|X-Y_m|^{2p}\right)^{\frac{p}{p-1}}=:A_m.
\end{align*}
We observe that $A_m\to A:=\left(E| X-X'|^{2p}\right)^{\frac{1}{p}} \left(p E|X-Y|^{2p}\right)^{\frac{p}{p-1}}$.
Finally, Fatou's lemma enables us to conclude  that
\begin{align*}
n\text{Var}(\mathcal{T}_p(P_n,Q))\leq n \lim\inf_m \text{Var}(\mathcal{T}_p(P_n,Q_m))\leq \lim\inf_m A_m =A.
\end{align*}
\end{proof}

\begin{Remark}\label{Sharper_assumptions}
As in Remark \ref{remark_var}, Corollary \ref{Lemma_tigh_p} yields the conclusion that 
$\{\sqrt{n}\left(\mathcal{T}_p(P_n,Q)-E \mathcal{T}_p(P_n,Q)\right)\}_{n\in \N}$ is tight
if $P$ and $Q$ have finite moments of order $2p$. 
This assumption is sharp in the sense that if $P$ is such that $\{\sqrt{n}\left(\mathcal{T}_p(P_n,Q)-E \mathcal{T}_p(P_n,Q)\right)\}_{n\in \N}$
for $Q=\delta_{\mathbf{0}}$ then $P$ must have finite moment of order $2p$. In fact, the optimal transport map from $P_n$ to $Q$ is $T(\mathbf{x})=\mathbf{0}$, hence, $ \mathcal{T}_p(P_n,Q)=\int |\mathbf{x}|^p dP_n(\mathbf{x})$ and 
\begin{align}\label{Classic_ctl}
\sqrt{n}\left( \mathcal{T}_p(P_n,Q)-E\mathcal{T}_p(P_n,Q) \right)={\textstyle \frac{1}{\sqrt{n}}\sum_{j=1}^n (|X_j|^p-E |X_1|^p)}.
\end{align}
It is well known (see, e.g., Chapter 10 in~{\em \cite{LedouxTalagrand}}) that the random variable in \eqref{Classic_ctl} is tight if and only if  
$E(|X_1|^{2p})<\infty$. Hence, as claimed, a finite moment of order $2p$ is a minimal requirement for $P$ to guarantee that 
$\{\sqrt{n}\left(\mathcal{T}_p(P_n,Q)-E \mathcal{T}_p(P_n,Q)\right)\}_{n\in \N}$ is tight for, say, every $Q$ with bounded support.
\end{Remark}
Condition \eqref{eq:tight_cond_1} is enough to achieve tightness with a cost $c$  satisfying assumptions (A1)-(A3).  In the following theorem we show that, with this assumptions on the cost, there exists a unique weak cluster point of the sequence $\{	\sqrt{n}\left(\mathcal{T}_c(P_n,Q)-E \mathcal{T}_c(P_n,Q)\right)\}_{n\in \N}$, which is Gaussian. Similar work, in the particular case of  the cost $|\cdot |^2$, was done in \cite{BaLo}, where a version of Efron-Stein inequality is used to prove that the empirical transport cost is approximately linear. This approach has also been used for the  entropic regularization of the empirical transport cost in~\cite{MeNi}. This tool based on Efron-Stein inequality requires to have some sort of uniform integrability, which can be guaranteed assuming finite moments of order $4+\delta$. Following arguments developed in~Remark \ref{Sharper_assumptions}, the following result proves that the moment assumption can be relaxed.
\begin{Theorem}\label{Theorem_vaiance_bound_general_sin_delta}
Assume $c(\mathbf{x},\mathbf{y})=h(\mathbf{x}-\mathbf{y})$ with $h$ differentiable and satisfying {\em (A1)-(A3)}.  Let $P, Q\in\mathcal{P}(\R^d)$ be such that $P\ll \ell_d$, $Q\ll \ell_d$, and $P$ has connected support and negligible boundary. Assume further that
\begin{align}\label{assumptio:c2}
\int h(2\mathbf{x})^2dP(\mathbf{x})<\infty \ \ \text{and} \ \ \int h(-2\mathbf{y})^2dQ(\mathbf{y})<\infty,
\end{align}
and \eqref{eq:tight_cond_1} holds. Then 
\begin{align}\label{conclusion1}
\sqrt{n}\left(\mathcal{T}_c(P_n,Q)-E \mathcal{T}_c(P_n,Q)\right)\stackrel{w}{\longrightarrow} N(0, \sigma^2_c(P,Q)), 
\end{align}
 where 
\begin{align}\label{sigma_def}
\sigma^2_c(P,Q):=\int{\varphi(\mathbf{x})^2}dP(\mathbf{x})-\left( \int{\varphi(\mathbf{x})}dP(\mathbf{x})\right)^2,
\end{align}
and $\varphi$ is an optimal transport potential for the cost $c$ from $P$ to $Q$.
\end{Theorem}
It should be noted at this point that the optimal transport potential in Theorem \ref{Theorem_vaiance_bound_general_sin_delta} is unique, up to the addition of a constant, as a consequence of Corollary \ref{cor:uniqueness}. It follows from the proof of Theorem \ref{Theorem_vaiance_bound_general_sin_delta} that $\varphi\in L_2(P)$. This implies that the limiting variance, $\sigma^2_c(P,Q)$, is well-defined and finite.

The proof of Theorem~\ref{Theorem_vaiance_bound_general_sin_delta} initially follows the path in~\cite{BaLo}. This means that we look at
\begin{align}\label{eq_def_R}
R_n:=\mathcal{T}_c(P_n,Q)-\int{\varphi(\mathbf{x})}dP_n(\mathbf{x}),
\end{align}
where $\varphi$ is an optimal transport potential from $P$ to $Q$ for the cost $c$. We write $R_n'$ for the version of $R_n$ computed from $X_1',X_2,\ldots,X_n$. Using the stability results for optimal transport potentials one can prove that $n(R_n-R'_n)\stackrel{a.s.}\longrightarrow 0$. 
If $n^2E(R_n-R'_n)^2\to 0$ then the conclusion in Theorem~\ref{Theorem_vaiance_bound_general_sin_delta} follows inmediately. Variance bounds obtained from the Efron-Stein inequality yield $n^2E(R_n-R'_n)^2\leq M$ under mild moment assumptions. However, the convergence $n^2E(R_n-R'_n)^2\to 0$ may fail wihtout some stronger assumptions (such as the $4+\delta$ moment assumption in~\cite{BaLo}). Our proof of Theorem~\ref{Theorem_vaiance_bound_general_sin_delta}
avoids these stronger assumptions by using the following workaround. First, the bound $n^2E(R_n-R'_n)^2\leq M$ and the Banach-Alaoglu Theorem (see, e.g., Theorem 3.16 in~\cite{Brezis}) show that, along subsequences, $n(R_n-R'_n)$ converges weakly to 0 in the Hilbert (hence reflexive) space $L_2(\mathbb{P})$. Then, the Banach-Saks property of Hilbert spaces (see, e.g., Exercise 5.34 in~\cite{Brezis}) shows that (taking further subsequences if necessary) there exists a Ces\`aro mean of $\{ n|R_{n}-R'_{n}|\}_{n\in \N}$ convergent to $0$ in $L^2(\mathbb{P})$ in the strong sense. We show then that 
the same holds with the Ces\`aro means of the sequence $\sqrt{n}(R_n-ER_n)$ and from this we conclude that $\sqrt{n}(R_n-ER_n)\to 0$ in probability, which yields, as a consequence, \eqref{conclusion1}. All the details are given in the proof postponed to the Appendix.

In general it is not possible to guarantee moment convergence in \eqref{conclusion1} under the minimal assumptions of Theorem~\ref{Theorem_vaiance_bound_general_sin_delta}.  The following theorem guarantees convergence of variances under slightly stronger assumptions.
\begin{Theorem}\label{Theorem_vaiance_bound_general}
Assume $c(\mathbf{x},\mathbf{y})=h(\mathbf{x}-\mathbf{y})$ with $h$ differentiable and satisfying {\em (A1)-(A3)}.  Let $P, Q\in\mathcal{P}(\R^d)$ be such that $P\ll \ell_d$, $Q\ll \ell_d$ and $P$ has connected support and negligible boundary. Suppose that \eqref{assumptio:c2} holds and assume $R_n$ is as in \eqref{eq_def_R}. Assume further that $X,X'$ and $Y$ are independent random variables with $X\sim P$, $X'\sim P$ and $Y\sim Q$.
If there exists some $\delta>0$ such that ,
\begin{equation}\label{eq:tight_cond2_delta}
\inf_{q_1,q_2\in[1,\infty]:\ \frac{1}{q_1}+\frac{1}{q_2}=1}\left[E| X-X'|^{(2+\delta)q_1}E|\nabla h(X-Y)|^{(2+\delta)q_2}\right] <\infty,\\
\end{equation}
then $n\text{\em Var}(R_n)\longrightarrow 0$. As a consequence,
\begin{align}\label{eq:desigualdad1}
n\text{\em Var}(\mathcal{T}_c(P_n,Q))\longrightarrow \sigma^2_c(P,Q).
\end{align}
\end{Theorem}

To get a more clear picture about the sharpness of the assumptions in Theorems \ref{Theorem_vaiance_bound_general_sin_delta} and 
\ref{Theorem_vaiance_bound_general}, we include the particular version for  potential costs, $c_p(\mathbf{x},\mathbf{y})=|\mathbf{x}-\mathbf{y}|^p$ for $p>1$ (recall from Remark~\ref{Remark:p_satisfies} that $c_p$ satisfies (A1)-(A3) for $p>1$).
\begin{Corollary}\label{Theo:centrallim_p}
Assume $p>1$. Let $P, Q\in\mathcal{P}(\mathbb{R}^d)$ be such that $P\ll \ell_d$ and has connected support and negligible boundary. If $P$ and $Q$ have finite moments of order $2p$, then
\begin{align}\label{eq:ctl_p}
\sqrt{n}\left(\mathcal{T}_p(P_n,Q)-E \mathcal{T}_p(P_n,Q)\right)\stackrel{w}{\longrightarrow} N(0, \sigma^2_p(P,Q)),
\end{align}
where 
\begin{align}\label{def:sigma_p}
 \sigma^2_p(P,Q):=\int{\varphi(\mathbf{x})^2}dP(\mathbf{x})-\left( \int{\varphi(\mathbf{x})}dP(\mathbf{x})\right)^2,
\end{align}
and $\varphi$ is an optimal transport potential from $P$ to $Q$ for $c_p$.
Moreover if $P$ has a finite moment of order $2p+\epsilon$ for some $\epsilon>0$, then
\begin{align}\label{second_claim_p}
	n\text{\em Var}(\mathcal{T}_p(P_n,Q))\longrightarrow \sigma^2_p(P,Q).
\end{align}
\end{Corollary}
\begin{proof}
A look at the proof of Corollary~\ref{Lemma_tigh_p} shows that finite $2p$ moments guarantee that \eqref{eq:tight_cond_1} holds. Clearly, \eqref{assumptio:c2} holds too, and we can apply Theorem \ref{Theorem_vaiance_bound_general_sin_delta} to conclude \eqref{eq:ctl_p} (the fact that absolute continuity of $Q$ is not necessary follows using the approximation argument in the proof of Corollary \ref{Lemma_tigh_p}). For \eqref{second_claim_p} we take in \eqref{eq:tight_cond2_delta} the conjugate pair $q_1=\frac{2p}{(p-1)2+\delta}$ and $q_2=\frac{q_1}{q_1-1}=\frac{2p}{2p-(2+\delta)}$, where $\delta=2-2p(1-\frac{1}{2p+\epsilon})$. With these choices  \eqref{eq:tight_cond2_delta} becomes
\begin{align*}
\begin{split}
\left(E| X_1-X'_1|^{(2p+\epsilon)}\right) \left(E\left(\int_{\R^d}  |X_1-\textbf{y}|^{2p} d Q(\textbf{y})\right)\right) <\infty,\\
\end{split}
\end{align*}
and we apply Theorem \ref{Theorem_vaiance_bound_general}. The case of  finite moment of order $2p+\epsilon$ for $Q$ follows similarly.
\end{proof}
\begin{Remark}\label{Sharper_assumptions2}{
As noted in  Remark~\ref{Sharper_assumptions}, the assumption of finite moments of order $2p$ (ar least for $P$) cannot be relaxed for tightness and, in that sense, the moment assumptions in Theorem \ref{Theo:centrallim_p} are sharp and cannot be improved. On the other hand, in the case $p=2$, Corollary \ref{Theo:centrallim_p} improves Theorem~4.1 in {\em \cite{BaLo}}, not only by proving that finite fourth moments are enough (the original assumption was finite moments of order $4+\epsilon$ in {\em \cite{BaLo}}), but also by assuming  milder regularity assumptions on $P$ and $Q$. In this new setting,  $P$ must have a connected support with a negligible boundary, relaxing the assumption of a convex support. The only price to pay is that variance convergence may fail under this relaxed assumptions.}
\end{Remark} 
So far we have considered CLTs for $\mathcal{T}_p(P,Q)$. Its  $p-$root $\mathcal{W}_p(P,Q):=(\mathcal{T}_p(P,Q))^{1/p}$ defines a well-known metric in the space of probabilities with finite moments of order $p$, the $p$-Wasserstein distance. Proving a CLT for the empirical Wasserstein distance is not a straightforward application of a delta-method and Corollary \ref{Theo:centrallim_p}, since we do not have a fixed centering constant in Theorem~\ref{Theo:centrallim2}. Yet, we can circunvent this issue and prove the following result.
\begin{Theorem}\label{lemma:Variance_bounds_dista}
Let $P\ne Q\in\mathcal{P}(\R^d)$ be such that $P\ll \ell_d$ and has connected support and negligible boundary.  Assume $P$ and $Q$ have finite moments of order $2p$  and $p>1$. Then, 
if $\sigma^2_p(P,Q)$ is defined as in Theorem \ref{Theo:centrallim_p},

\begin{align*}
\sqrt{n}\left(\mathcal{W}_p(P_n,Q)-\left(E [\mathcal{W}^p_p(P_n,Q)]\right)^{\frac{1}{p}}\right)\stackrel{w}{\longrightarrow} N(0, \beta^2_p(P,Q)), 
\end{align*}
where $\beta^2_p(P,Q):=\left(\frac{1}{p\mathcal{W}^p_p(P_n,Q)^{p-1}}\right)^2\sigma^2_p(P,Q)$.
\end{Theorem}
\begin{proof}
Setting 
\begin{equation*}
A_n:=\mathcal{W}_p(P_n,Q)\ \ \text{and} \ \ B_n:=\left(E [\mathcal{W}_p^p(P_n,Q)]\right)^{\frac{1}{p}},
\end{equation*}
we know from Corollary \ref{Theo:centrallim_p} that 
\begin{equation}\label{eq:ctl_p2}
\sqrt{n}\left(A_n^p-B_n^p\right)\stackrel{w}{\longrightarrow} N(0, \sigma^2_p(P,Q)).
\end{equation}
Moreover, the bound
$$
\mathcal{W}_p^p(P_n,Q) \leq 2^{p-1}\int  |\textbf{x}|^pdP_n(\textbf{x})+2^{p-1}\int |\textbf{y}| ^pd Q(\textbf{y}),
$$
together with the assumption of finite moments of order $2p$, imply that $\mathcal{W}_p^p(P_n,Q) $ is uniformly integrable.  It follows that 
\begin{equation}\label{limitsAn}
 \text{$A_n\xrightarrow{a.s.}\mathcal{W}_p(P,Q)$,  and $B_n\rightarrow \mathcal{W}_p(P,Q)$}.
\end{equation}
By the mean value theorem applied to the function $t\mapsto t^p$,  there exists $\varepsilon_n\in (0,1)$ such that 
\begin{equation}\label{ValueMed}
A_n^p-B_n^p=(A_n-B_n)p(A_n\varepsilon_n+B_n(1-\varepsilon_n))^{p-1}.
\end{equation}
The limits of  \eqref{limitsAn}  imply that necessarily $p(A_n\varepsilon_n+B_n(1-\varepsilon_n))^{p-1}\xrightarrow{a.s.} p \mathcal{W}_{p}(P,Q)^{p-1}>0$.  This fact, together with the limit \eqref{eq:ctl_p2} and  Slutsky's theorem applied in \eqref{ValueMed} conclude the proof.
\end{proof}

\subsection{Two-sample case}
For $n,m\in \N$ let $X_1,\dots, X_n $ and $Y_1,\dots, Y_m $ be independent i.i.d. random samples with distributions $P$ and $Q$. Consider the correspondent empirical measures $P_n:=\frac{1}{n}\sum_{k=1}^{n}\delta_{X_k}$ and $ Q_m:=\frac{1}{m}\sum_{k=1}^{m}\delta_{Y_k}$. 
At first sight one may conjecture that the approach leading to  Theorems \ref{Theorem_vaiance_bound_general_sin_delta} and 
\ref{Theorem_vaiance_bound_general} trivially extends to the two-sample setup, yielding a CLT for $\mathcal{T}_c(P_n,Q_m)$. However, a closer look at the
proof shows that major issues appear when extending Claim 3. For this reason an adaptation of Theorem~\ref{Theorem_vaiance_bound_general_sin_delta} to the two-sample setup is left for further work. On the other hand, under stronger moment assumptions, such as \eqref{eq:tight_cond2_delta}, the extension is straightforward. We present the result avoiding additional details.
\begin{Theorem}\label{Theo:centrallim2}
Assume $c(\mathbf{x},\mathbf{y})=h(\mathbf{x}-\mathbf{y})$ with $h$ differentiable and satisfying (A1)-(A3). Let $P, Q\in\mathcal{P}(\mathbb{R}^d)$ be such that $P\ll \ell_d$, $Q\ll\ell_d$ and both have connected support and negligible boundary. Assume that \eqref{assumptio:c2} holds and also that there exists some $\delta>0$ such that \eqref{eq:tight_cond2_delta} holds, as well as the corresponding conditions exchanging the roles of $P$ and $Q$. Then, if $\frac{nm}{n+m} \rightarrow \lambda \in (0,1)$ as $n,m\rightarrow \infty$,  
\begin{align*}
\sqrt{\textstyle \frac{nm}{n+m}}\left(\mathcal{T}(P_n, Q_m)-E\mathcal{T}(P_n, Q_m)\right)\stackrel{w}{\longrightarrow} N\left(0,(1-\lambda)\sigma^2_c(P,Q)+\lambda\sigma^2_c(Q,P)\right),
\end{align*}
with $\sigma^2_c(Q,P)$ as in \eqref{sigma_def}. Furthermore, 
$$n\text{\em Var} (\mathcal{T}(P_n, Q_m))\to (1-\lambda)\sigma^2_c(P,Q)+\lambda\sigma^2_c(Q,P).$$
\end{Theorem}

\section{Appendix} \label{s:append}
\subsection{Proofs of main results}


\textsc{Proof of Theorem~\ref{Theo:main2}.} We prove each claim separately.
To prove \textit{(i)} we take, without loss of generality, $\mathbf{p}_0\in \text{supp}\cap \text{dom}(\nabla \psi)$  as in \eqref{potential.fixing} (hence, $\psi_{n}(\mathbf{p}_0)=0$). From \eqref{eq:cota2} and Lemma \ref{Lem:technic3} we see that for each compact $K\subset \text{supp}(Q)$ there exists a subsequence $\psi_{n_k}$ and a constant $R=R(K)>0$ such that for $\mathbf{a},\mathbf{x}\in K$
\begin{align*}
|\psi_{n_k}(\mathbf{x})-\psi_{n_k}(\mathbf{a})|\leq | \mathbf{x}-\mathbf{a} |R.
\end{align*}
Hence, the functions of the sequence $\{ \psi_{n_k} \}$ are $R$-Lipschitz on each compact set and $\psi_{n_k}(\mathbf{p}_0)=0$ and we can apply Arzel\`a-Ascoli theorem in each compact set to conclude that there exits a continuous function $f$ such that $\psi_{n_{k_m}}\rightarrow f$ uniformly on the compact sets of $\text{supp}(Q)$ for some subsequence.\\
 We claim that $ f=\psi +C$. To prove it we consider $\mathbf{x}\in\text{supp}(Q)$ and any sequence $\mathbf{y}_{n}\in \partial^c\psi_{n}(\mathbf{x})$, by Lemma~\ref{Lem:technic3} we know that there exist a sub-sequence $\{\mathbf{y}_{n_k}\}_{k\in \N}$ which is bounded. Hence, by Lemma~\ref{Teo:main1}, there exists $\mathbf{y}\in \partial\psi(\mathbf{x})$ such that $\mathbf{y}_{n_k}\rightarrow\mathbf{y}\in \partial^c \psi(\mathbf{x})$ along a subsequence. We keep the notation for this sub-sequence and note that it satisfies $$\psi_{n_k}(\mathbf{z})\leq \psi_{n_k}(\mathbf{x})+[c(\mathbf{z},\mathbf{y}_{n_k})-c(\mathbf{x},\mathbf{y}_{n_k})] \ \ \text{for all}\ \ \mathbf{z}\in\R^d,$$ and by taking limits, 
 $$f(\mathbf{z})\leq f(\mathbf{x})+[c(\mathbf{z},\mathbf{y})-c(\mathbf{x},\mathbf{y})] \ \ \text{for all}\ \ \mathbf{z}\in \text{dom}(f).$$
Therefore, $\partial^c f(\mathbf{x})$ is non-empty for every $\mathbf{x}\in\text{supp}(Q)$. This entails that $f$ is $c-$concave and, as a consequence, almost surely differentiable.
Moreover, $\mathbf{y}\in\partial^c f(\mathbf{x})\cap \partial^c \psi(\mathbf{x})$. We conclude that $\nabla^c f =\nabla^c \psi$ a.s. in $\text{supp}(Q)$ and \textit{(i)} follows by Corollary \ref{cor:uniqueness}.\\ \\
We turn now to \textit{(ii)} and assume, on the contrary, that there exists a sequence $\{\mathbf{x}_n\}\subset K \text{  and  } \mathbf{y}_n\in \partial^c \psi_n(\mathbf{x}_n)$ such that
\begin{align}\label{eq:there:exist_seq}
|\mathbf{y}_n-\nabla^c \psi(\mathbf{x}_n) |>\epsilon \text{  for some  } \epsilon>0 \text{  and  all } n.
\end{align}
Compactness of $K$ implies that there exists $\mathbf{x}\in K$ such that $\mathbf{x}_{n}\rightarrow \mathbf{x}$ along a subsequence, which, to ease notation, we  denote also as $\mathbf{x}_n$. Lemma~\ref{Lem:technic3} implies that $\mathbf{y}_n$ also converges to some $\mathbf{y}$ along a subsequence. But then Theorem~\ref{Teo:main1} shows that $\mathbf{y}=\nabla^c \psi(\mathbf{x})$ which contradicts \eqref{eq:there:exist_seq}.
\hfill $\Box$
}
 
\bigskip 
{
\textsc{Proof of Theorem~\ref{Theorem_vaiance_bound_general_sin_delta}.} 
We write $(X_1',\ldots,X_n')$ for an independent copy of $(X_1,\ldots,X_n)$ and denote by $P_n^{(i)}$ the empirical measure
on $(X_1, \dots,X_i' ,$ $\dots, X_n)$. As in \eqref{eq_def_R},
$$R_n=\mathcal{T}_c(P_n,Q)-\int{\varphi(\mathbf{x})}dP_n(\mathbf{x}),$$
where $\varphi$ is an optimal transport potential from $P$ to $Q$.
We write $R_n^{(i)}$ for the version of $R_n$ computed from $P_n^{(i)}$ instead of $P_n$.
To ease notation it will be convenient to write $P_n'$ rather that $P_n^{(1)}$ and $R_n'$ instead of $R_n^{(1)}$ at some points.}

The guideline of the proof is to show that $n(R_n-R'_n)\stackrel{a.s.}\longrightarrow 0$ and $n^2E(R_n-R'_n)^2\leq M$. From this we can obtain, using the Banach-Alaoglu theorem and the Banach-Saks property (see details below), that there exists a Ces\`aro mean of $\{ n|R_{n}-R'_{n}|\}_{n\in \N}$ convergent to $0$ in $L^2(\mathbb{P})$. Finally the same holds with the Ces\`aro means of the sequence $\sqrt{n}(R_n-ER_n)$. To conclude we will prove that these three claims imply the central limit theorem. We follow this path in the following complete proof, which we split into three main steps: \\
\\{
\textbf{Claim 1:} $n(R_n-R'_n)\stackrel{a.s.}\longrightarrow 0$ and $n^2E(R_n-R'_n)^2\leq M$. 

We write $\varphi_n$ for an optimal transport potential between $P_n$ and $Q$. Since 
\begin{eqnarray*}
\mathcal{T}_c(P'_n,Q)&=&\sup_{(f,g)\in \Phi_c(P,Q)}\int f(\textbf{x}) dP'_n(\textbf{x})+\int g(\textbf{y}) dQ(\textbf{y})\\
&\geq& \int \varphi_n(\textbf{x}) dP'_n(\textbf{x})+\int \varphi^c_n(\textbf{y}) dQ(\textbf{y}),
\end{eqnarray*}
then we have
\begin{align*}
R_n'\geq \frac{1}{n}\varphi_n(X'_1)+\frac{1}{n}\sum_{k=2}^{n}\varphi_n(X_1)-\frac{1}{n}\sum_{k=2}^{n}\varphi(X_k)-\frac{1}{n}\varphi(X'_1)+\int \varphi_n^c(\textbf{y}) dQ(\textbf{y}).
\end{align*}
This implies that
\begin{align}\label{eq:bound_lemma1}
R_n-R_n'\leq \frac{1}{n}\big(\varphi_n(X_1)-\varphi(X_1)-\varphi_n(X'_1)+\varphi(X'_1)\big).
\end{align}
By Theorem \ref{Theo:main2} we can assume, without loss of generality, that, almost surely, $\varphi_n\to\varphi$, uniformly on compact subsets of $\text{supp}(P)$. This entails that $n(R_n-R'_n)_+\stackrel{a.s.}\to 0$. By symmetry, 
$n(R_n'-R_n)_+\stackrel{a.s.}\to 0$ and we conclude that $n(R_n'-R_n)\stackrel{a.s.}\to 0$.

For the second part of this claim we recall that 
\begin{align*}
n(R_n-R_n')=n(\mathcal{T}_c(P_n,Q)-\mathcal{T}_c(P'_n,Q))-(\varphi(X_1)-\varphi(X_1')).
\end{align*}
It follows from \eqref{eq:tight_cond_1} and the proof of Lemma~\ref{Lemma:var_bound} that $n^2 E(\mathcal{T}_c(P_n,Q)-\mathcal{T}_c(P'_n,Q))^2$ is a bounded sequence and, therefore, it suffices to show that $E\varphi(X_1)^2<\infty$. 
To check this, we fix $\mathbf{x}_0\in \text{supp}(P)\cap\text{dom}(\nabla \varphi)$. From \eqref{eq:cota} we get that 
\begin{align*}
|\varphi(X_1)|&\leq |\varphi(\mathbf{x}_0)|+|c(X_1,\mathbf{y})-c(\mathbf{x}_0,\mathbf{y})|+|c(X_1,\mathbf{b})-c(\mathbf{x}_0,\mathbf{b})|,\\
&\leq |\varphi(\mathbf{x}_0)|+c(X_1,\mathbf{y})+c(\mathbf{x}_0,\mathbf{y})+c(X_1,\mathbf{b})+c(\mathbf{x}_0,\mathbf{b}),
\end{align*}
{for all}   $(\mathbf{x}_0,\mathbf{b}),(X_1,\mathbf{y})\in \partial^c \varphi$.
Since $\varphi$ is differentiable at $\mathbf{x}_0$ then if $X_1\in \text{dom}(\nabla \varphi)$ we have 
\begin{eqnarray*}
|\varphi(X_1)|&\leq& |\varphi(\mathbf{x}_0)|+c(X_1,\nabla^c\varphi(X_1))+c(\mathbf{x}_0,\nabla^c\varphi(X_1))\\
&&+c(X_1,\nabla^c\varphi(\mathbf{x}_0))+c(\mathbf{x}_0,\nabla^c\varphi(\mathbf{x}_0)).
\end{eqnarray*}
Recalling that $c(\mathbf{x},\mathbf{y})=h(\mathbf{x}-\mathbf{y})$ and that $h$ is convex, we see that
\begin{align*}
c(X_1,\nabla^c\varphi(X_1))=h(X_1-\nabla^c\varphi(X_1))\leq{\textstyle \frac{1}{2}} h( 2 X_1)+{\textstyle \frac{1}{2}}h(-2 \nabla^c\varphi(X_1)).
\end{align*}
Hence, using the fact that $Q=\nabla^c\varphi \# P$ and \eqref{assumptio:c2} we deduce that 
\begin{align*}
E(c(X_1,\nabla^c\varphi(X_1)))^2\leq \int h( 2 \mathbf{x})^2dP(\mathbf{x})+\int h(-2 \mathbf{y})^2dQ(\mathbf{y})<\infty.
\end{align*}
Similarly, we check that $E(c(X_1,\nabla^c\varphi(\mathbf{x}_0))^2)<\infty$ and $E(c(\mathbf{x}_0,\nabla^c\varphi(X_1))^2)<\infty$. This shows that $\varphi(X_1)$ has a finite second moment, as claimed.

}

{
\textbf{Claim 2:} From every subsequence of $\{ n|R_{n}-R'_{n}|\}_{n\in \N}$ we can extract a subsequence for which the Ces\`aro mean converges to $0$ in $L^2(\mathbb{P})$.

From Claim 1 and the Banach-Alaoglu theorem (see Theorem 3.16 in \cite{Brezis}) applied on the Hilbert space $L^2(\mathbb{P})$, we see that, along subsequences, $n|R_n-R'_n|\stackrel{L^2}\rightharpoonup 0 $, where $\stackrel{L^2}\rightharpoonup $ denotes the weak convergence in the space $L^2(\mathbb{P})$. By a theorem of Banach and Saks (see the Banach–Saks property, exercise 5.24 in \cite{Brezis}), we conclude that there exists a sub-sequence, $\{ n_k|R_{n_k}-R'_{n_k}|\}_{k\in \N}$,
such that the Ces\`aro means 
\begin{align}\label{Cesaro_sums_1}
g_m= \frac{1}{m}\sum_{k=1}^m n_k|R_{n_k}-R'_{n_k}|
\end{align}
converge strongly to $0$ in $L^2(\mathbb{P})$, that is,
\begin{align}\label{Cesaro_sums_strong}
E\Big( \frac{1}{m}\sum_{k=1}^m n_k|R_{n_k}-R'_{n_k}|\Big)^2\longrightarrow 0.
\end{align}
\textbf{Claim 3:} From every subsequence of $\sqrt{n}(R_n-ER_n)$ we can extract a further subsequence for which the Ces\`aro mean converges to $0$ in $L^2(\mathbb{P})$.

There exists a Ces\`aro mean of $\sqrt{n}(R_n-ER_n)$ convergent to $0$ in $L^2(\mathbb{P})$.
\\For ease of notation, we write  $k$ instead of $n_k$ in \eqref{Cesaro_sums_1}. We set  $G_m:=\frac{1}{m}\sum_{k=1}^m \sqrt{k}R_k$. By the Efron-Stein inequality  
\begin{align}\label{efron_st_reduced}
 \text{Var}(G_m)\leq  \frac{1}{2}\sum_{i=1}^m E(G_m-G_m^{(i)})^2.
\end{align}
Next, we observe that
\begin{align*}
E(G_m-G_m^{(i)})^2&=E\Big(\frac{1}{m}\sum_{k=1}^m \sqrt{k}(R_k-R_k^{(i)})\Big)^2\\
&=\frac{1}{m^2}\sum_{k=1}^m k E\Big( R_k-R_k^{(i)}\Big)^2\\
&+\frac{2}{m^2}\sum_{k=1}^m \sum_{j=k+1}^m\sqrt{k}\sqrt{j}E( R_k-R_k^{(i)}) (R_j-R_j^{(i)}),
\end{align*}
since for the terms with $k < i$ the difference is is $0$. Hence
\begin{align*}
E(G_m-G_m^{(i)})^2&=\frac{1}{m^2}\sum_{k=i}^m k E\left( R_k-R_k^{(i)}\right)^2\\
&+\frac{2}{m^2}\sum_{k=i}^m \sum_{j=k+1}^m \sqrt{k}\sqrt{j}E( R_k-R_k^{(i)}) (R_j-R_j^{(i)})\\
&=\frac{1}{m^2}\sum_{k=i}^m k E\left( R_k-R'_k\right)^2\\
&+\frac{2}{m^2}\sum_{k=i}^m \sum_{j=k+1}^m \sqrt{k}\sqrt{j}E( R_k-R'_k) (R_j-R'_j).
\end{align*}
Here, the second equality comes from the fact that $\left( R_k-R'_k\right)^2$ has the same distribution as $( R_k-R_k^{(i)})^2$ when $i\leq k$, and the same happens with $(R_k-R'_k) (R_j-R'_j)$ and $( R_k-R_k^{(i)}) (R_j-R_j^{(i)})$. Now turning back to \eqref{efron_st_reduced} we have
\begin{eqnarray*}
 \text{Var}(G_m)&\leq&  \frac{1}{2}\frac{1}{m^2}\sum_{i=1}^m\sum_{k=i}^m k E\left( R_k-R'_k\right)^2\\&&+\frac{2}{m^2}\frac{1}{2}\sum_{i=1}^m\sum_{k=i}^m \sum_{j=k+1}^m \sqrt{k}\sqrt{j}E( R_k-R'_k) (R_j-R'_j) \\
 &\leq &  \frac{1}{2}\frac{1}{m^2}\sum_{i=1}^m\sum_{k=i}^m k E\left( R_k-R'_k\right)^2\\
 &&+\frac{1}{m^2}\sum_{i=1}^m\sum_{k=i}^m \sum_{j=k+1}^m \sqrt{k}\sqrt{j} E|( R_k-R'_k)|| (R_j-R'_j)|\\
 &=& \frac{1}{2}\frac{1}{m^2}\sum_{k=1}^m k^2 E\left( R_k-R'_k\right)^2\\
 &&+\frac{1}{m^2}\sum_{i=1}^m\sum_{k=i}^m \sum_{j=k+1}^m \sqrt{k}\sqrt{j} E|( R_k-R'_k)|| (R_j-R'_j)|.
\end{eqnarray*}
Compute the last term to obtain
\begin{eqnarray*}
\lefteqn{\sum_{i=1}^m\sum_{k=i}^m \sum_{j=k+1}^m \sqrt{k}\sqrt{j} E|( R_k-R'_k)|| (R_j-R'_j)|}\hspace*{2cm}\\
&=&\sum_{j=1}^m\sum_{k=1}^{j-1}\sum_{i=1}^{k}  \sqrt{k}\sqrt{j} E|( R_k-R'_k)|| (R_j-R'_j)|\\
&=&\sum_{j=1}^m\sum_{k=1}^{j-1}\sqrt{k}\sqrt{j} E|( R_k-R'_k)|| (R_j-R'_j)|k\\
&\leq&\sum_{j=1}^m\sum_{k=1}^{j-1}kj E|( R_k-R'_k)|| (R_j-R'_j)|.
\end{eqnarray*} 
We conclude that
\begin{align*}
 \text{Var}(G_m)&\leq \frac{1}{2}\frac{1}{m^2}\sum_{k=1}^m k^2 E\left( R_k-R'_k\right)^2+\sum_{j=1}^m\sum_{k=1}^{j-1}kj E|( R_k-R'_k)|| (R_j-R'_j)|\\
 &= \frac{1}{2} E\Big(\frac{1}{m} \sum_{k=1}^m k|R_k-R'_k|\Big)^2,
\end{align*}
which, together with \eqref{Cesaro_sums_strong}, shows that
\begin{align}\label{cesaro_sums_2}
E\Big(\frac{1}{m}\sum_{k=1}^m \sqrt{k}(R_k-ER_k)\Big)^2=\text{Var}(G_m)\longrightarrow 0.
\end{align}
Finally  we have proven that for every subsequence of $\{ G_m\}_{m \in \N}$ we can find a further subsequence converging to $0$ strongly in $L^2(\mathbb{P})$, and Claim 3 follows.

Now we are ready to prove the central limit theorem. Note that by the Central Limit Theorem we have
\begin{align*}
\int{\varphi(\mathbf{x})}dP_n(\mathbf{x})-E\left(\int{\varphi(\mathbf{x})}dP_n(\mathbf{x})\right)\stackrel{w}\longrightarrow N(0, \sigma^2_c(P,Q)).
\end{align*}
As a consequence, the Ces\`aro means converge to the same limit,
\begin{align}\label{ctl_class}
\frac{1}{m}\sum_{k=1}^{m}\Big\lbrace\int{\varphi(\mathbf{x})}dP_k(\mathbf{x})-E\Big(\int{\varphi(\mathbf{x})}dP_k(\mathbf{x})\Big)\Big\rbrace\stackrel{w}\longrightarrow N(0, \sigma^2_c(P,Q)).
\end{align}
Both \eqref{ctl_class} and \eqref{cesaro_sums_2} imply that 
\begin{align}\label{ctl_cesaro}
\frac{1}{m}\sum_{k=1}^{m}\sqrt{k}\left\lbrace\mathcal{T}_c(P_k,Q)-E\mathcal{T}_c(P_k,Q)\right\rbrace\stackrel{w}\longrightarrow N(0, \sigma^2_c(P,Q)).
\end{align}
The variance bound of Lemma~\ref{Lemma:var_bound} and Remark~\ref{remark_var} yield tightness of $\sqrt{n}\{ \mathcal{T}_c(P_n, $ $Q)-E\mathcal{T}_c(P_n,Q)\}_{n\in \N}$. Hence each sub-sequence has a convergent sub-sequence to some limiting distributions, say $\gamma$. The Ces\`aro means must converge also to $\gamma$. Finally, from \eqref{ctl_cesaro} we conclude that $\gamma= N(0, \sigma^2_c(P,Q))$ and the proof follows.

\hfill $\Box$}

\bigskip 
{
\textsc{Proof of Theorem~\ref{Theorem_vaiance_bound_general}.} 
We keep the same notations as in the proof of Theorem~\ref{Theorem_vaiance_bound_general_sin_delta}, noting that the new assumption \eqref{eq:tight_cond2_delta} has no influence on the proof of Claim 1. Hence, we only have to prove that $n^2(R_n-R_n')_+^2$ is uniformly integrable and, in fact, recalling that
\begin{align*}
n(R_n-R_n')=n(\mathcal{T}_c(P_n,Q)-\mathcal{T}_c(P'_n,Q))-(\varphi(X_1)-\varphi(X_1'))
\end{align*}
and that $\varphi(X_1)$ has a finite second moment (as shown in the proof of Theorem~\ref{Theorem_vaiance_bound_general_sin_delta}), it suffices to prove uniform integrability of $n(\mathcal{T}_c(P_n,Q)-\mathcal{T}_c(P'_n,Q))$.

To check this we denote $Z:=\mathcal{T}_c(P_n,Q)$ and $Z':=\mathcal{T}_c(P'_n,Q) $. Arguing as in the proof of Lemma \ref{Lemma:var_bound} we see that 
\begin{align*}
(Z-Z')_+\leq | X_1-X'_1| \int_{C'_1}  |\nabla h(X_1-\textbf{y})|d Q(\textbf{y}).
\end{align*}
Hence, by H\"older's inequality, for every pair $(q_1,q_2)\in\alpha$ it holds that
\begin{eqnarray}\nonumber
\lefteqn{E(n(Z-Z')_+)^{2+\delta}\leq E\Big\lbrace| X_1-X'_1|^{2+\delta} \Big(\int_{C'_1}  n|\nabla h(X_1-\textbf{y})|d Q(\textbf{y})\Big)^{2+\delta}\Big\rbrace}\hspace*{1cm}\\
\nonumber
&&\leq \left(E| X_1-X'_1|^{(2+\delta)q_1}\right)^{\frac{1}{q_1}}\Big( E \Big(\int_{C'_1} n |\nabla h(X_1-\textbf{y})|d Q(\textbf{y})\Big)^{(2+\delta)q_2}\Big)^{\frac{1}{q_2}}.
\end{eqnarray}
A further use of H\"older's inequality yields that
\begin{eqnarray*}
\lefteqn{\textstyle \int_{C'_1}  |\nabla h(X_1-\textbf{y})|d Q(\textbf{y})}\hspace*{1cm}\\
&\leq  &\textstyle  \left(\int_{C'_1} d Q(\textbf{y})\right)^{\frac{(2+\delta)q_2-1}{(2+\delta)q_2}}\left(\int_{C'_1}  |\nabla h(X_1-\textbf{y})|^{(2+\delta)q_2} d Q(\textbf{y})\right)^{\frac{1}{(2+\delta)q_2}}
\\
&=&{\textstyle \frac{1}{n^{\frac{(2+\delta)q_2-1}{(2+\delta)q_2}}}}\Big({\textstyle\int_{C'_1}}  |\nabla h(X_1-\textbf{y})|^{q_2(2+\delta)} d Q(\textbf{y})\Big)^{\frac{1}{(2+\delta)q_2}}
\end{eqnarray*}
Note that $(X_1', \dots, X_n)$ is independent of $X_1$, hence, the same holds for $C'_k$, for $k=1,\dots, n$. By exchangeability, we have that $ \int_{C'_1}  |\nabla h(X_1-\textbf{y})|^{(2+\delta)q_2} d Q(\textbf{y})$ is equally distributed as $ \int_{C'_k}  |\nabla h(X_1-\textbf{y})|^{(2+\delta)q_2} d Q(\textbf{y})$, $k=2, \dots, n$. This implies 
\begin{align*}\textstyle
E\left\lbrace \int_{C'_1}  |\nabla h(X_1-\textbf{y})|^{(2+\delta)q_2} d Q(\textbf{y})\right\rbrace&= \textstyle\frac{1}{n}E\left\lbrace \sum_{i=1}^n \int_{C'_i}  |\nabla h(X_1-\textbf{y})|^{(2+\delta)q_2} d Q(\textbf{y})\right\rbrace \\
&\textstyle\leq { \frac{1}{n}}E\left\lbrace \int_{\R^d}  |\nabla h(X_1-\textbf{y})|^{(2+\delta)q_2} d Q(\textbf{y})\right\rbrace,
\end{align*}
which, in turn, entails
\begin{align*}
\textstyle E\left( \int_{C'_1}  |\nabla h(X_1-\textbf{y})| d Q(\textbf{y})\right)^{(2+\delta)q_2}\leq \frac{1}{n^{(2+\delta)q_2}}E\left(\int_{\R^d}  |\nabla h(X_1-\textbf{y})|^{(2+\delta)q_2} d Q(\textbf{y})\right).
\end{align*}
Combining the last estimates, we can see that
\begin{align*}\textstyle
E(n(Z-Z')_+)^{(2+\delta)} \leq  \left(E| X_1-X'_1|^{(2+\delta)q_1}\right)^{\frac{1}{q_1}} \left(E\left(\int_{\R^d}  |\nabla h(X_1-\textbf{y})|^{(2+\delta)q_2} d Q(\textbf{y})\right)\right)^{\frac{1}{q_2}}
\end{align*}
and the proof follows.

\hfill $\Box$
}

\bigskip 
{
\textsc{Proof of Theorem~\ref{Theo:centrallim2}.} We set
$$R_{n,m}:=\mathcal{T}_c(P_n,Q_m)-\int{\varphi(\mathbf{x})}dP_n(\mathbf{x})-\int{\psi(\mathbf{y})}dQ_m(\mathbf{y})$$
with $\varphi$ an optimal transport potential from $P$ to $Q$ for the cost $c$ and $\psi=\varphi^c$
and observe that it suffices to show that $\frac{nm}{n+m}\text{Var}(R_{n,m})\to 0$.
Once again the key of the proof is  Efron-Stein's inequality. Note that $R_{n,m}$ as a function of $X_1,\dots, X_n, Y_1,\dots, Y_m$ is symmetric in its $n$ first variables as well as in the last $m$. Let $X_1'$(resp. $Y_1'$) be a copy of $X_1$ (resp. $Y_1$) both independent of $X_1,\dots, X_n, Y_1,\dots, Y_m$, finally let $P_n'$ (resp. $Q'_n$) be the empirical distribution of $X'_1,X_2\dots, X_n$ (resp. $Y'_1,Y_2\dots, Y_m$). Hence, if we denote\begin{align*}
R'_{n,m}&:=\mathcal{T}_c(P'_n,Q_m)-\int{\varphi(\mathbf{x})}dP'_n(\mathbf{x})-\int{\psi(\mathbf{y})}dQ_m(\mathbf{y}), \\
R''_{n,m}&:=\mathcal{T}_c(P_n,Q'_m)-\int{\varphi(\mathbf{x})}dP_n(\mathbf{x})-\int{\psi(\mathbf{y})}dQ'_m(\mathbf{y}),
\end{align*} 
by the Efron-Stein inequality we have
\begin{eqnarray*}
\frac{nm}{n+m}\text{Var}(R_{n,m})\leq \frac{n^2m }{n+m} E(R_{n,m}-R'_{n,m})_+^2 + \frac{ n m^2}{n+m} E(R_{n,m}-R''_{n,m})_+^2 .
\end{eqnarray*}
Now, to conclude, it suffices to prove that 
\begin{eqnarray}
\label{eq:asymptotic_double1}
& n^2E((R_{n,m}-R'_{n,m})_+^2)\to  0 \quad \text{and}\\ 
& m^2E((R_{n,m}-R''_{n,m})_+^2)\to  0.
\end{eqnarray}
We handle \eqref{eq:asymptotic_double1}, which will follow if we prove that $n(R_{n,m}-R'_{n,m})_+\to 0$ \ a.s. and also that 
$n^2(R_{n,m}-R'_{n,m})^2$ is uniformly integrable. For the first claim note note that if $\varphi_n$ (resp. $\psi_n$) is an optimal transport potential from $P_n$ to $Q_m$ (resp. from $Q_m$ to $P_n$) then
\begin{eqnarray*}
\lefteqn{R'_{n,m} }\hspace*{0cm} \\
&\geq & \textstyle \int{\varphi_n(\mathbf{x})}dP'_n(\mathbf{x})+\int{\psi_m(\mathbf{y})}dQ_m(\mathbf{y})-\int{\varphi(\mathbf{x})}dP'_n(\mathbf{x})-\int{\psi(\mathbf{y})}dQ_m(\mathbf{y}).
\end{eqnarray*}
As a consequence,
\begin{align*}
R_{n,m}-R'_{n,m}&\leq \int_{\R^d}{\left(\varphi_n(\mathbf{x})-\varphi(\mathbf{x})\right)}\left(dP_n(\mathbf{x})-dP'_n(\mathbf{x})\right)\\
&=\frac{1}{n}\left(\varphi_n(X_1)-\varphi(X_1)-\varphi_n(X'_1)+\varphi(X'_1)\right)
\end{align*}
and we see that 
\begin{align}\label{eq:inequality_Rmn}
n\left( R_{n,m}-R'_{n,m}\right)_+\leq |\varphi_n(X_1)-\varphi(X_1)-\varphi_n(X'_1)+\varphi(X'_1)|.
\end{align}
By Theorem~\ref{Theo:main2}, with a right choice of potentials we can guarantee that, $P-$a.s., $\varphi_n\to\varphi$
and conclude that $n\left( R_{n,m}-R'_{n,m}\right)_+\to 0$ $P-$a.s. \\
Finally, it only remains to prove that $n^2E\left( R_{n,m}-R'_{n,m}\right)_+^2$ is uniformly bounded, which follows arguing as in the proof of Theorem~\ref{Theorem_vaiance_bound_general}.

\hfill $\Box$
}

\subsection{Proofs of Lemmas} 


{
\textsc{Proof of Lemma~\ref{Teo:main1}.}
Set $\mathbf{x}_0\in  \text{dom}(\nabla^c\psi)\cap \mbox{Supp}(Q)$ and {\color{blue} $\mathbf{y}_0=\nabla^c\psi(\mathbf{x}_0)$}. By Lemma~\ref{lemma:cont_c} we see that for each $\epsilon>0$ there exists some $\delta>0$ such that if $|\mathbf{z}-\mathbf{x}_0|<\delta$ then $\partial^c\psi(\mathbf{z})\subset B(\mathbf{y}_0, \epsilon)$. 
Let $\pi$ be the unique optimal transport plan between $Q$ and $P$. By Theorem~\ref{Theo:GaMc} $\text{supp}(\pi)\subset \partial^c\psi$.  This entails
\begin{align*}
\pi\left( B(\mathbf{x}_0, \delta)\times B(\mathbf{y}_0, \epsilon)\right)=\pi\left( B(\mathbf{x}_0, \delta)\times \R^d\right)=Q(B(\mathbf{x}_0, \delta))=\eta>0,
\end{align*}
where the inequality comes from the assumption $Q\ll\ell_d$. Repeating the argument with a decreasing sequence $\epsilon_k \rightarrow 0$, we obtain a sequence $\delta_k\leq \frac{1}{k}$ such that 
\begin{eqnarray*}
\pi\left( B(\mathbf{x}_0, \delta_k)\times B(\mathbf{y}_0, \epsilon_k)\right)&=&\pi\left( B(\mathbf{x}_0, \delta_k)\times \R^d\right)\\&=&Q(B(\mathbf{x}_0, \delta_k))>\eta_k>0.
\end{eqnarray*}
Let $\pi_n$ be an optimal transport plan between $P_n$ and $Q_n$. We observe that 
\begin{enumerate}
\item[(a)] $\pi_n\stackrel{w}{\rightarrow} \pi$ by Theorem 5.20 in \cite{Vi},
\item[(b)] $\text{supp}(\pi_n)\subset \partial^c\psi_n$ by Theorem~\ref{Theo:GaMc}.
\end{enumerate}

By (a) there exists $N_{k}$ such that, for $n\geq N_k$, $\pi_n(B(\mathbf{x}_0, \delta_k)\times B(\mathbf{y}_0, \epsilon_k))\geq \eta_k/2$. Hence, by (b) we can choose a pair 
\begin{align}\label{eq:extract_sub}
(\mathbf{x}_{n_k}, \mathbf{y}_{n_k})\in \partial^c\psi_n \cap \left( B(\mathbf{x}_0, \delta_k)\times B(\mathbf{y}_0, \epsilon_k)\right).
\end{align}
\\
As a consequence of \eqref{eq:extract_sub}, since $\epsilon_k,\delta_k\rightarrow 0$, we can extract a sub-sequence of $(\mathbf{x}_n,\mathbf{y}_n)\in \partial^c \psi_n$ converging to $(\mathbf{x}_0,\mathbf{y}_0)$. Define $a_n:=\psi_n(\mathbf{x}_n)-\psi(\mathbf{x}_0)$ and $\tilde{\psi}_n:=\psi_n-a_n$ (which has the same $c$-superdifferential as $\psi_n$). Now, \eqref{eq:extract_sub} implies that $\partial^c \tilde{\psi}_n$ are $c$-cyclically monotone sets which do not escape to the horizon. By Theorem~\ref{Teo:Horizon} and Lemma~\ref{lemma:convergence_c} we deduce that $\partial^c \tilde{\psi}_n$ converges to a $c$-cyclically monotone set $\Gamma$ along a sub-sequence.  Necessarily $\Gamma\subset \partial^c f$ for some $c$-concave function $f$. We observe that $(\mathbf{x}_0,\mathbf{y}_0)\in \partial^c f$. If we take another arbitrary point $\mathbf{x}\in \text{dom}(\nabla^c\psi) $ and $\partial^c \psi (\mathbf{x})=\{ \mathbf{y} \}$, we can apply the same arguments to check that  $(\mathbf{x},\mathbf{y})\in \partial^c f$. Hence, $\text{dom}(\nabla^c\psi) \subset \text{dom}(f)$. Since $f$ is differentiable a.s then $\partial^c f$ is a singleton a.s and, therefore, that $\nabla f =\nabla \psi$ a.s. in the support of $Q$, which is connected. Using Theorem \ref{theo:iqual} we conclude that there exists a constant $C$ such that $\psi=f-C$ in $\Omega$. Hence $\partial^c\psi =\partial^c f$ and the result follows.  \hfill $\Box$}

\bigskip
{
\textsc{Proof of Lemma~\ref{Lem:Boundness}.} We can assume, without loss of generality, that $\mathbf{p}$ is in the interior of the domain of $\psi$, since otherwise the result is trivial. With this assumption, we check first that we cannot have $\psi_n(\mathbf{p}_n)\to\infty$. In fact, in that case, by $c$-concavity we would have
$$\psi_n(\mathbf{p}_n)\leq c(\mathbf{p}_n,\mathbf{y})-\psi_n^c(\mathbf{y})$$
for all $\mathbf{y}$. Hence, we would have $\psi_n^c(\mathbf{y_n})\to -\infty$ for all $y_n\to \mathbf{y}$. Now, take $\mathbf{p}_0$ as in \eqref{potential.fixing}. By Lemma \ref{Teo:main1}  we  can choose $(\tilde{\mathbf{p}}_n,\mathbf{y}_n)$ with $\mathbf{y}_n\in\partial^c \psi_n (\tilde{\mathbf{p}}_n)$, $\tilde{\mathbf{p}}_n\to \mathbf{p}_0$ and $\mathbf{y}_n\to \nabla^c \psi (\mathbf{p}_0)=\mathbf{y}_0$. But then we would have 
$\psi_n(\tilde{\mathbf{p}}_n)\to\psi({\mathbf{p}}_0)=0$, while, on the other hand,
$$\psi_n(\tilde{\mathbf{p}}_n)=c(\tilde{\mathbf{p}}_n,\mathbf{y}_n)-\psi_n^c(\mathbf{y}_n)\to \infty,$$
which is a contradiction.

Now, we can assume, taking subsequences if necessary, that $\psi_n(\mathbf{p}_n)<-n$ for all $n\in \N$.  Now, taking $\mathbf{y}_n\in \R^d\in \partial^c \psi(\mathbf{p}_n)$
we have that
\begin{eqnarray}\label{ineq:defin_c_co}
\psi_n(\mathbf{x})\leq c(\mathbf{x},\mathbf{y}_n) +\lambda_n, \ \ \text{for all $\mathbf{x}\in \R^d$},
\end{eqnarray}
where $ \lambda_n=\psi_n(\mathbf{p}_n) -c(\mathbf{p}_n,\mathbf{y}_n)$. Hence, by assumption we have that $c(\mathbf{p}_n,\mathbf{y}_n) +\lambda_n\leq-n $ for all $n \in \N$. Now, let $\{ \mathbf{x}_n\}$ be a bounded sequence such that $\psi_n(\mathbf{x}_n)$ is bounded. Then
\begin{align*}
\psi_n(\mathbf{x}_n)< c(\mathbf{x}_n,\mathbf{y}_n) -c(\mathbf{p}_n,\mathbf{y}_n) -n.
\end{align*}
Since $\psi_n(\mathbf{x}_n),\mathbf{p}_n, \mathbf{x}_n$ are bounded, then $|\mathbf{y}_n|\rightarrow\infty$. For each $n$ we choose the height $r_n\in [0, \infty]$ and the direction $\mathbf{z}_n$ of the largest cone with vertex $\mathbf{p}_n-\mathbf{y}_n$ such that
$$K\big(r_n,{\textstyle  \frac{\pi}{1+r_n^{-1}}}, \mathbf{z}_n,\mathbf{p}_n-\mathbf{y}_n\big)\subset \left\lbrace \mathbf{x}:\  h(\mathbf{x})\leq h(\mathbf{p}_n-\mathbf{y}_n) \right\rbrace.$$
Since $\mathbf{z}_n\in \mathbb{S}_{d-1}$, then up to a sub-sequence, we can assume that $\mathbf{z}_n\rightarrow \mathbf{z}\in \mathbb{S}_{d-1}$. Also, since $|\mathbf{p}_n-\mathbf{y}_n|\rightarrow\infty$ then the condition (A2) implies that  $r_n\rightarrow\infty$ (note that otherwise if $|r_n|<R$ then (A2) is no longer true for $r=R+1$ and $\theta=\frac{\pi}{1+r_n^{-1}}$).

Now let $\{ \mathbf{x}_n\}_{n\in \N }\subset \subset  \{\mathbf{x}:\left\langle \mathbf{z},\mathbf{x}-\mathbf{p}\right\rangle>0 \}$ be a bounded sequence. From the fact that $r_n\to\infty$ we see that 
$$ \cos\left({\textstyle \frac{1}{2}\frac{\pi}{1+r_n^{-1}}}\right)\rightarrow 0.$$
Therefore, for big enough $n$
\begin{align}\label{cota_r_n}
|\mathbf{x}_n-\mathbf{p}_n|\cos\left({\textstyle \frac{1}{2}\frac{\pi}{1+r_n^{-1}}}\right)<\left\langle \mathbf{z},\mathbf{x}-\mathbf{p}\right\rangle<r_n.
\end{align}
As a consequence  $\mathbf{x}_n\in K(r_n, \frac{\pi}{1+r_n^{-1}}, \mathbf{z},\mathbf{p}_n)$, which implies that
\begin{align*}
\mathbf{x}_n-\mathbf{y}_n\in K\big(r_n,{\textstyle \frac{\pi}{1+r_n^{-1}}}, \mathbf{z},\mathbf{p}_n-\mathbf{y}_n\big)\subset \left\lbrace \mathbf{x}:\  h(\mathbf{x})\leq h(\mathbf{p}_n-\mathbf{y}_n) \right\rbrace.
\end{align*}
From this we conclude that $c(\mathbf{x}_n,\mathbf{y}_n)\leq c(\mathbf{p}_n,\mathbf{y}_n)$, and turning back to \eqref{ineq:defin_c_co}, that
\begin{align*}
\psi_n(\mathbf{x}_n)\leq c(\mathbf{x}_n,\mathbf{y}_n)+\lambda_n\leq c(\mathbf{p}_n,\mathbf{y}_n)+\lambda_n \leq -n,
\end{align*}
and the proof follows.
\hfill $\Box$}

\bigskip
{
\textsc{Proof of Lemma~\ref{Lem:technic3}.} We split the proof into the following steps:

\textbf{Step 1} {(Pointwise boundedness)}: Fix $\mathbf{x}\in \text{supp}(Q){\cap \text{dom}(\psi)}$. By Lemma~\ref{Teo:main1} there exists a $c$-cyclically monotone set $\Gamma$ such that, up to taking sub-sequences, $\partial^c\psi_n\rightarrow \Gamma$ in the sense of Painlev\'e-Kuratowski. Hence, there exists a sequence $(\mathbf{x}_{n_k},\mathbf{y}_{n_k})\in\partial^c\psi_{n_k}$ satisfying
$$(\mathbf{x}_{n_k},\mathbf{y}_{n_k})\rightarrow (\mathbf{x},\mathbf{y})\in \Gamma .$$  
Assume $ \{\psi_n(\mathbf{x}_{n_k})\}_{k \in \N}$ is not bounded. Then there exist a sub-sequence $ \psi_{n_{k_m}}(\mathbf{x}_{n_{k_m}})\rightarrow -\infty$ (the case 
$ \psi_{n_{k_m}}(\mathbf{x}_{n_{k_m}})\rightarrow \infty$ can be excluded arguing as at the beginning of the proof of Lemma \ref{Lem:Boundness}. Now, we take $\mathbf{p}_0$ as in \eqref{potential.fixing} and observe that, 
\begin{align}\label{eq:take_limits}
0 \leq \psi_{n_{k_m}}(\mathbf{x}_{n_{k_m}})+c(\mathbf{p}_0,\mathbf{y}_{n_{k_m}})-c(\mathbf{x}_{k_m},\mathbf{y}_{n_{k_m}}).
\end{align}
Taking limits as $m \rightarrow \infty$ in \eqref{eq:take_limits} leads to a contradiction. Hence, the sequence $\{\psi_{n_k}(\mathbf{x}_{n_k})\}_{k\in \N}$ 
must be bounded.

For ease of reading we will use the same notation for the subsequence $\{\psi_{n_k}\}_{k \in \N}$ and the main sequence $\{\psi_{n}\}_{n \in \N}$ in the subsequent steps 2 and 3.\\ \\
\textbf{Step 2} (For every compact $K\subset \text{supp}(Q)$ there exists $M>0$ such that $|\psi_{n}(K)|\leq M$ for large enough $n$): Assume, on the contrary, that for every $m\in \N$ there exists some $n_m\in \N$ such that $\mathbf{k}_{n_m}\in K $ and $|\psi_{n_m}(\mathbf{k}_{n_m})|>m$. Then $|\psi_{n_m}(\mathbf{k}_{n_m})|\rightarrow \infty$ as $m\rightarrow\infty$ and, by compactness, $\mathbf{k}_{n_m}\rightarrow \mathbf{k}\in K$ along a subsequence. By Lemma~\ref{Lem:Boundness} we see that there exists $\mathbf{z}\in \R^d$ such that $\psi_n(\mathbf{x}_n)$ is not bounded, for every bounded sequence $\{ \mathbf{x}_n\}\subset \subset\{\mathbf{x}:\left\langle \mathbf{z},\mathbf{x}-\mathbf{k}\right\rangle>0 \}$. Now take $\mathbf{x}_0\in \text{supp}(Q) \cap\{\mathbf{x}:\left\langle \mathbf{z},\mathbf{x}-\mathbf{k}\right\rangle>0 \}$. Since this last set is open, there exists $\varepsilon>0$ such that $B(\mathbf{x}_0,\epsilon)\subset \subset \text{supp}(Q)\cap\{\mathbf{x}:\left\langle \mathbf{z},\mathbf{x}-\mathbf{k}\right\rangle>0 \}$,  and this contradicts Step 1 applied to the point $\mathbf{x}_0$. 
\\ \\
\textbf{Step 3} (For every compact $K\subset \text{supp}(Q)$ there exists $M>0$ such that $\partial^c \psi_{n}(K)\subset B(\mathbf{0},M)$ for large enough $n$): Assume this fails for a compact $K\subset \text{supp}(Q)$. Since $ \text{supp}(Q)$ is open, there exists $\epsilon>0$ such that $$K^{\epsilon}:=\{\mathbf{x}: d(\mathbf{x},K)\leq \epsilon \}\subset \subset \text{supp}(Q).$$ 
By Step 2 there exists $M>0$ and $n_0\in\N$ such that $|\psi_n(\mathbf{k})|\leq M$, for all $\mathbf{k}\in K^{\epsilon}$ and $n\geq n_0$. Now we can take $\{ \mathbf{k}_n\}_{n \in \N}\subset K$ and $\mathbf{y}_n\in\partial^c\psi_n(\mathbf{k}_n)$ such that $|\mathbf{y}_n|\rightarrow\infty$, define $\mathbf{v}_n:=\mathbf{k}_n-\mathbf{y}_n$ and observe that for $n$ big enough $|\mathbf{v}_n|>1$. Define $\xi_n:=1-\epsilon \frac{1}{|\mathbf{v}_n|}$ and note that $\xi_n\rightarrow 1$. All $\mathbf{k}_n$ belong to the compact set $K$, hence define $\mathbf{z}_n :=\mathbf{k}_n+(\xi_n-1)\mathbf{v}_n=\mathbf{k}_n+\frac{\epsilon}{2|\mathbf{v}_n|}\mathbf{v}_n\in K^{\epsilon}$, for which we can ensure $\psi_n(\mathbf{z}_n)>-M$. By definition of superdifferentials we have
\begin{align*}
2M\geq \psi_n(\mathbf{z}_n)-\psi_n(\mathbf{k}_n)\geq h(\mathbf{v}_n)-h(\xi_n \mathbf{v}_n)
\end{align*}
and by convexity of $h$ there exists $\mathbf{s}_n\in \partial h (\xi_n \mathbf{v}_n)$, for which we have
\begin{align}\label{eq:ecuacion:bound_M}
2M\geq \left\langle (1-\xi_n)\mathbf{v}_n,\mathbf{s}_n \right\rangle =\epsilon \left\langle{\textstyle  \frac{\mathbf{v}_n}{|\mathbf{v}_n|}},\mathbf{s}_n \right\rangle.
\end{align}
Observe that we also have
\begin{align*}
h(\mathbf{0})\geq h(\xi_n \mathbf{v}_n)+\left\langle \mathbf{0}-\xi_n \mathbf{v}_n,\mathbf{s}_n \right\rangle.
\end{align*}
Now, since $\xi_n>1-\epsilon>0$ and $|\mathbf{v}_n|\to\infty$ we have $|\xi_n \mathbf{v}_n|\rightarrow \infty$ and, consequently,
\begin{align}\label{eq_limit_abs}
\liminf_{n\rightarrow \infty} \left\langle {\textstyle \frac{\mathbf{v}_n}{|\mathbf{v}_n|}},\mathbf{s}_n \right\rangle \geq \liminf_{n\rightarrow \infty} {\textstyle \frac{h(\xi_n \mathbf{v} )}{|\xi_n \mathbf{v}_n |}} \rightarrow \infty,
\end{align}
with the last limit following from the condition (A3). This contradicts \eqref{eq:ecuacion:bound_M}.

\hfill $\Box$}

\bigskip

{
\textsc{Proof of Lemma~\ref{Lemma:var_bound}.} We write $X_1'$ for a random variable with law $P$ and independent from $(X_1,\ldots,X_n)$. Denote by $P'_n$ the empirical measure associated to $(X_1',X_2, \dots, X_n)$ and $Z':=\mathcal{T}_c(P'_n,Q)$. Since $Q\ll \ell_d$, there exists an optimal transport map from $Q$ to $P_n'$, which we denote by $T$. We set  $$C'_1:=\{ \mathbf{y}\in \R^d:  T(\mathbf{y})= X_1' \}, \quad C'_i:=\{ \mathbf{y}\in \R^d:  T(\mathbf{y})= X_i \},\, i\geq 2,$$
and observe that $Q(C'_i)=\frac{1}{n}$ and 
\begin{align*}
Z'&= \int c(\textbf{x},\textbf{y}) d \pi'(\textbf{x}, \textbf{y})= \int_{C'_1} c(X'_1,\textbf{y}) d Q(\textbf{y}) +\sum_{i=2}^n\int_{C'_i} c(X_i,\textbf{y}) d Q(\textbf{y}),\\
Z&\leq \int_{C'_1} c(X_1,\textbf{y}) d Q(\textbf{y})+ \sum_{i=2}^n \int_{C'_i} c(X_i,\textbf{y}) d Q(\textbf{y}).
\end{align*}
From this we see that (recall that $h(X-\cdot)$ is convex and $Q$-a.s. differentiable) 
\begin{eqnarray*}
Z-Z'&\leq & \int_{C'_1} \left( c(X_1,\textbf{y})- c(X'_1,\textbf{y}) \right) d Q(\textbf{y})\\
&\leq &  \int_{C'_1}  \left\langle \nabla h(X_1-\textbf{y}), X_1-X'_1\right\rangle d Q(\textbf{y})\\
&\leq & | X_1-X'_1| \int_{C'_1}  |\nabla h(X_1-\textbf{y})|d Q(\textbf{y}).
\end{eqnarray*}
Hence, by H\"older's inequality, for any pair $(q_1,q_2)\in\alpha$,
\begin{align}\label{eq:holder1}
\begin{split}
E(Z-Z')_+^2&\leq E\Big\lbrace| X_1-X'_1|^2 \Big(\int_{C'_1}  |\nabla h(X_1-\textbf{y})|d Q(\textbf{y})\Big)^2\Big\rbrace\\
&\leq \Big(E| X_1-X'_1|^{2q_1}\Big)^{\frac{1}{q_1}}\Big( E \Big(\int_{C'_1}  |\nabla h(X_1-\textbf{y})|d Q(\textbf{y})\Big)^{2q_2}\Big)^{\frac{1}{q_2}}.
\end{split}
\end{align}
Using again H\"older's inequality we get that
\begin{align*}
\int_{C'_1}  |\nabla h(X_1-\textbf{y})|d Q(\textbf{y})&\leq \Big(\int_{C'_1} d Q(\textbf{y})\Big)^{\frac{2q_2-1}{2q_2}}\Big(\int_{C'_1}  |\nabla h(X_1-\textbf{y})|^{2q_2} d Q(\textbf{y})\Big)^{\frac{1}{2q_2}}\\
&={\textstyle\frac{1}{n^{\frac{2q_2-1}{2q_2}}}}\Big(\int_{C'_1}  |\nabla h(X_1-\textbf{y})|^{2q_2} d Q(\textbf{y})\Big)^{\frac{1}{2q_2}}.
\end{align*}
Finally, by exchangeability,
\begin{align*}
E\Big\lbrace \int_{C'_1}  |\nabla h(X_1-\textbf{y})|^{2q_2} d Q(\textbf{y})\Big\rbrace&= \frac{1}{n}E\Big\lbrace \sum_{i=1}^n \int_{C'_i}  |\nabla h(X_1-\textbf{y})|^{2q_2} d Q(\textbf{y})\Big\rbrace \\&= \frac{1}{n}E\Big\lbrace \int_{\R^d}  |\nabla h(X_1-\textbf{y})|^{2q_2} d Q(\textbf{y})\Big\rbrace,
\end{align*}
which implies that
\begin{align*}
E\Big( \int_{C'_1}  |\nabla h(X_1-\textbf{y})| d Q(\textbf{y})\Big)^{2q_2}\leq \frac{1}{n^{2q_2}}E\Big(\int_{\R^d}  |\nabla h(X_1-\textbf{y})|^{2q_2} d Q(\textbf{y})\Big).
\end{align*}
Combining the last estimates with \eqref{eq:holder1} leads to 
\begin{align*}
E(Z-Z')_+^2 \leq \frac{1}{n^2} \Big(E| X_1-X'_1|^{2q_1}\Big)^{\frac{1}{q_1}} \Big(E\Big(\int_{\R^d}  |\nabla h(X_1-\textbf{y})|^{2q_2} d Q(\textbf{y})\Big)\Big)^{\frac{1}{q_2}}.
\end{align*}

\hfill $\Box$

}

\end{document}